\documentclass{article}
\usepackage{new-mysymb}

\newcommand\CORR[1]{{\color{black}{#1}}}

\newcommand\bx{\bar{x}}

\newcommand{\EEA}{\end{eqnarray}}
\newcommand{\BEA}{\begin{eqnarray}}
\newcommand{\comment}[1]{}

\newtheorem{thm}{Theorem}[section]
\newtheorem{prop}[thm]{Proposition}

\newtheorem{example}[thm]{Example}

\newtheorem{rem}[thm]{Remark}
\usepackage{booktabs}
\usepackage{authblk}

\title{Learning Coarse-Grained Dynamics on Graph}
\author[1]{Yin Yu\thanks{Department of Aerospace Engineering, yzy5368@psu.edu}}
\author[1]{John Harlim\thanks{Department of Mathematics, Department of Meteorology and Atmospheric Science, \& Institute for Computational and Data Sciences, jharlim@psu.edu}}
\author[1]{Daning Huang\thanks{Corresponding author, Department of Aerospace Engineering,  daning@psu.edu}}
\author[1]{Yan Li\thanks{Department of Electrical Engineering, yql5925@psu.edu}}
\affil[1]{The Pennsylvania State University, University Park, PA, 16802}

\date{}

\begin{document}

\maketitle

\begin{abstract}
We consider a Graph Neural Network (GNN) non-Markovian modeling framework to identify coarse-grained dynamical systems on graphs. Our main idea is to systematically determine the GNN architecture by inspecting how the leading term of the Mori-Zwanzig memory term depends on the coarse-grained interaction coefficients that encode the graph topology. Based on this analysis, we found that the appropriate GNN architecture that will account for $K$-hop dynamical interactions has to employ a Message Passing (MP) mechanism with at least $2K$ steps.
We also deduce that the memory length required for an accurate closure model decreases as a function of the interaction strength under the assumption that the interaction strength exhibits a power law that decays as a function of the hop distance. Supporting numerical demonstrations on two examples, a heterogeneous Kuramoto oscillator model and a power system, suggest that the proposed GNN architecture can predict the coarse-grained dynamics under fixed and time-varying graph topologies. 
\end{abstract}

{\bf Keywords:} networked dynamics, coarse-graining, graph neural network, non-Markovian modeling, time-varying topology

\section*{Nomenclature}

{\renewcommand\arraystretch{1.0}
\noindent\begin{longtable}{@{}l @{\quad=\quad} l@{}}
$A$, $\alpha_i$      & Coefficients for autonomous dynamics, and the component for node $i$ \\
$B$, $\beta_{ij}$    & Interaction strengths, and the component between nodes $i$ and $j$ \\
$D_i$                & Number of features at a node in the $i$th layer of the neural network \\
$d$                  & Average number of 1-hop neighbors of all nodes \\
$\cE_t$, $\bar{\cE}_t$ & Edges of a graph at time $t$, and the coarse-grained version \\
$F$, $f_{ij}$        & Interaction dynamics, and the component between nodes $i$ and $j$ \\
$F_G$                & A function represented by graph neural network \\
$f$, $f_i$           & Autonomous dynamics, and the component for node $i$ \\
$f_E$, $f_P$, $f_D$  & Encoder, processor, and decoder of the neural network \\
$\cG_t$, $\bar{\cG}_t$ & Graph at time $t$, and the coarse-grained version \\
$H$, $h_i^{(k)}$     & Features in neural network, and the $i$th component at layer $k$ \\
$I_1$                & Memory integral of leading order non-Markovian term \\
$J_k$                & Integer index for graph at time $k$ \\
$K$                  & Total number of hops \\
$\cL$                & Infinitesimal Koopman generator \\
$L$                  & Lipschitz constant \\
$\vL$, $\tilde{\vL}$ & Weighted graph Laplacian, and its transformed version \\
$M$                  & Number of node groups \\
$m_j$                & Number of resolved states of group $j$ \\
$\cN_{i,t}^{(k)}$    & The $k$-hop neighbors of node $i$ at time $t$ \\
$\cN_{i,t}^{[k]}$    & The nodes that are exactly $k$-hop away from node $i$ at time $t$ \\
$\bar{\cN}$          & Neighbors after coarse-graining \\
$N$                  & Number of nodes \\
$N_C$                & Number of graph convolution layers \\
$N_x$                & Number of states of entire graph \\
$n_i$                & Number of states at node $i$ \\
$\cP$, $\cQ$         & Projection operators \\
$P$, $Q$             & Matrices for coordinate projection \\
$R$                  & 2-Norm of $\Phi$ \\
$r^{(1)}$, $r^{(2)}$ & Resolved and unresolved dynamics after projection \\
$S$                  & Order of one graph convolution layer \\
$s_j$                & Number of unresolved states of group $j$ \\
$T$                  & Length of memory \\
$\cV$, $\cV_j$       & All nodes of a graph, and the $j$th group \\
$\bar{\cV}$          & All nodes of a coarse-grained graph \\
$\bar{\cW}$, $\bar{w}_{ij}$ & Edge weights of $\bar{\cG}$, and the $(i,j)$th component \\
$x$, $x_i$           & State vector, and the component for node $i$ \\
$\bar{x}$, $\bar{x}_j$   & Resolved states, and the component for group $j$ \\
$x'$, $x_j'$         & Unresolved states, and the component for group $j$ \\
$\Delta t$           & Time step size \\
$\epsilon$           & Decay rate of interaction strength \\
$\kappa$             & Interaction strength coefficient in Kuramoto oscillators \\
$\vtQ$               & Learnable parameters of a neural network \\
$\sigma$             & Activation function in neural network \\
$\Phi$, $\Phi_{ij}$  & Basis for resolved states, and its $(i,j)$th component \\
$\Psi$, $\Psi_{ij}$  & Basis for unresolved states, and its $(i,j)$th component \\
$|\cdot|$            & The number of elements in a set; matrix $\infty$-norm \\
$|\cdot|_2$          & Matrix 2-norm \\
$\|\cdot\|$          & Norm in Banach space $C(\cA,\bR^m)$ \\
$\Box^+$             & Left inverse \\
$\overline{\Box^\Phi}$                 & Projected component \\
$\widetilde{\Box^\Phi}$, $\Box^\Psi$   & Residual components \\
$[\Box]_i$, $[\Box]_{ij}$              & The $i$th or $(i,j)$th component \\
$\Box^{11}$, $\Box^{12}$               & Coefficients after coarse-graining
\end{longtable}}

\section{Introduction}

Dynamical systems on graphs have many natural applications in modeling complex systems such as power grid systems \cite{Yu2024}, particle systems in physics \cite{battaglia2016interaction}, chemical reaction predictions \cite{do2019graph}, fluid dynamic simulation \cite{belbute2020}, and many other applications where the data are represented as graphs. To facilitate such a modeling paradigm, Graph Neural Networks (GNN) have emerged as the state-of-the-art approaches (see the comprehensive reviews in \cite{wu2020comprehensive,zhou2020graph}). 

While graph representation is an effective way to suppress the scaling of complex dynamical systems, the computational cost of training such models is significant for large graphs and/or when the state dimension of each node is high. What is more daunting is the availability of the data to train such a GNN whose parameters only increase as functions of the network size and state dimension. Since the system is usually partially observed, where the observable is a coarse-grained variable, then it is more relevant to identify a reduced-order model induced by the coarse-graining of the graphs to predict the dynamics of the observable.

The main contribution of this paper is toward the design of appropriate GNN architecture for the reduced-order model when the underlying system exhibits $k$-hop interactions, i.e., having interactions between nodes that are $k-1$ nodes apart. Since the graph topology is encoded by the interaction coefficients of the underlying full system, the design of the reduced-order model should, in principle, encode the interaction coefficients of the coarse-grained graph dynamics as inputs. To facilitate this information in GNN, we employ the Mori-Zwanzig formalism \cite{zwanzig:73,stinis2007higher} to verify that the reduced-order model, or specifically, the leading order approximation of the Mori-Zwanzig memory term, depends quadratically on the interaction coefficients of the coarse-grained graph dynamics. Instead of approximating the Mori-Zwanzig memory kernels as in the classical literature (e.g., \cite{harlim2015parametric,lei2016data}), which is a difficult task and is problem-dependent, we use this quadratic dependence information to determine the number of layers and the polynomial order in the Message Passing (MP) mechanism of GNN, which corresponds to the number of hops in the interaction.
Understanding such a dependence also allows us to deduce the relation between the memory length in the non-Markovian reduced-order model to the interaction strength, under the assumption that the interaction coefficient for $k$-hop neighbors is of order $\epsilon^k$. In particular, we found that the delay embedding time parameter of the reduced-order model that incorporates $2k$-hop in their GNN architecture is short (or long) for systems with strong (or weak) interaction strength. We support this theoretical estimate with a numerical validation on an academic testbed, the Kuramoto model \cite{kuramoto1975international}.

Furthermore, we will demonstrate that our approach can predict the dynamic of the observable under topological changes of the graphs, which has practical significance in applications. For example, in the modern power grid, the complex dynamical system is characterized by multiple microgrids of interconnected loads and Distributed Energy Resources (DERs). In this application, it is important to understand the system's sensitivity caused by topology changes, which allows the controller to make appropriate decisions to either contribute energy to the grid or sustain autonomous operation during power outages or in regions devoid of grid connectivity. 

With power system application in mind, in addition to verifying the approach on the Kuramoto model, we will demonstrate the GNN model on a microgrid consisted of 5 generator buses and 5 load buses. We will verify our hypothesis that the system interaction strength decays as a function of hops and demonstrate the accuracy of the proposed GNN in estimating the prediction of the bus voltage under time-varying topologies. 

The remainder of this paper is organized as follows. In Section~\ref{sec2}, we discuss the canonical formulation for dynamics on graphs and give two relevant examples, the Kuramoto model and a microgrid power system. In Section~\ref{sec3}, we employ the Mori-Zwanzig formalism to verify that the leading-order expansion of the non-Markovian term has a quadratic dependence on the pairwise interaction coefficients. In Section~\ref{sec:gnn}, we describe the detail of the GNN model facilitated by the analysis in Section~\ref{sec3}.
In Section~\ref{sec5}, we present numerical simulations of the proposed GNN on the two examples. We supplement the paper with two appendices: one that reports the details of the power grid model and another one that reports the detailed coarse-graining calculations.

\section{Dynamics on Graphs}\label{sec2}
In this section, we provide the notations for dynamics on graphs and discuss two relevant examples that will be numerically investigated in this paper.

We first define the graph structure that may have time-varying topology.  At time $t$, let $\cG_t=(\cV,\cE_t)$ be a graph with a set of $N$ nodes, $\cV\subset\bN^+$, and a set of edges $\cE_t\subseteq\cV\times\cV$.  The graph may contain self-connections, e.g., for node $i$, $(i,i)\in\cE_t$.  The graph may contain directed connections between nodes $i$ and $j$, i.e., $(i,j)\in\cE_t$ and $(j,i)\notin\cE_t$.  By convention, $(i,j)$ means $i$ pointing to $j$.

For node $i$, define its 0-hop neighbor as itself $\cN_i^{(0)}=\{i\}$, and its 1-hop neighbors at time $t$ as $\cN_{i,t}^{(1)}=\{k|(k,i)\in\cE_t\}\cup \{i\}$.  For conciseness, denote $\cN_{i,t}=\cN_{i,t}^{(1)}$.  The number of 1-hop neighbors, $\abs{\cN_{i,t}}$, is referred to as the degree of node $i$, $d_i$.  The average degree of all nodes is denoted by $d$.
The $k$-hop neighbors at time $t$ are defined as $\cN_{i,t}^{(k)}=\bigcup_{n\in\cN_{i,t}^{(k-1)}}\cN_{n,t}$ for $k\geq 1$.
At time $t$, define the set of nodes that are exactly $k$-hop away from node $i$ as $\cN_{i,t}^{[k]}=\{m|m\in\cN_{i,t}^{(k)}\wedge m\notin\cN_{i,t}^{(k-1)}\}$.  Clearly, $\cN_{i,t}^{(k)}=\bigcup_{j=0}^k\cN_{i,t}^{[j]}$.  When the graph is randomly connected, so that $d_i=O(d)$, $\abs{\cN_{i,t}^{[k]}}=O(d^k)$.

Next, we define a dynamical system on a time-varying graph $\cG_t$.  Suppose node $i$ has states $x_i\in\cX_i\subset\bR^{n_i}$, and the state dimensions can be different among the nodes.  Denote all states of the graph dynamics
$
    x = [x_1,x_2,\cdots,x_N]\in\cX\subset\bR^{N_x}
$,
where $\cX=\cX_1\times\cX_2\times\cdots\times\cX_N$ and $N_x=\sum_{i=1}^N n_i$.  Furthermore, let the dynamic at node $i$ be governed by the following system of ODEs,
\begin{equation}\label{eqn_nod}
    \dot{x}_i = \alpha_if_i(x_i) + \sum_{j\in\cN_{i,t}^{(K)}} \beta_{ij}f_{ij}(x_i,x_j), 
\end{equation}
where $\alpha_i\in\bR^{n_i\times n_i}$, $\beta_{ij}\in\bR^{n_i\times n_j}$. Here, we assume that the functions $f_i:\cX_i\mapsto\bR^{n_i}$
and $f_{ij}:\cX_i\times\cX_j\mapsto\bR^{n_j}$ to be $C^1$  on a compact forward-invariant set, $\mathcal{A} \subset \mathcal{X}$ \cite{mauroy2016global,mauroy2020introduction} (e.g., if this is a measure preserving dynamics, $\mathcal{A}$ can denote the compact support of the invariant measure). In \eqref{eqn_nod}, the first term accounts for the autonomous dynamics (with respect to $x_i$) and the summation accounts for the pair-wise interactions among up to $K$-hop neighbors.

The full-order graph dynamics is denoted compactly as
\begin{equation}\label{eqn_dyn}
    \dot{x} = Af(x) + B(t) \otimes F(x,x),
\end{equation}
where 
\[A = \begin{pmatrix}\alpha_1 \\ & \alpha_2  \\ & & \ddots \\ & & & \alpha_N \end{pmatrix}\]
is a matrix of size $N_x\times N_x$ and the following tensors are defined as,
\[
B(t) = (\beta_{ij}(t) \in \bR^{n_i\times n_j })_{i,j=1}^{N}, \quad F(x,x) = (f_{ij}(x_i,x_j) \in \bR^{n_j})_{i,j=1}^{N},
\]
with $\beta_{ij}=0$ and $f_{ij}=0$ if $j\notin\cN_{i,t}^{(K)}$ and $\beta_{ii}=0$ for any $i$.
The notation $\otimes$ is to denote the sum of the ``element-wise'' product between the matrix $\beta_{ij}$ and vector $f_{ij}$; in \eqref{eqn_dyn} the sum is effectively over $K$-hop neighbors $j\in\cN^{(K)}_{i,t}$. Therefore, it is clear that $B(t)\otimes F(x,x) \in \bR^{N_x}$.
The system of ODEs in \eqref{eqn_dyn} is nothing but a non-autonomous system with affine control $\{\beta_{ij}(t)\}$.


\begin{rem}
    Conventionally, two nodes are directly connected if they interact.  But in the current definition, two nodes that are $k$-hop apart may still interact.  We introduce this characteristic to study the impact of distant coupling on graph dynamics.
\end{rem}

\begin{example}[Kuramoto model]\label{eg:kuramoto_1}
The Kuramoto model describes the coupling among a network of oscillators (i.e., nodes), and in its classical version \cite{kuramoto1975international} each node interacts with all other nodes.  In this study, we consider a modified version \cite{Filatrella2008,Menara2019}, where the nodes only interact with their 1-hop neighbors.  The governing equation is
\BEA
\dot{\theta}_i = \omega_i + \sum_{j\in\cN_{i,t}} \kappa_{ij}\sin(\theta_j-\theta_i), \label{kuramoto}
\EEA
where the interaction strength between nodes $i$ and $j$ is given by $\kappa_{ij}(t)> 0$.  The neighborhood $\cN_{i,t}$ is determined by the time-varying graph topology, an instance of which will be provided in the numerical examples. 

For example, if the topology represented in Figure~\ref{kuramoto_topo}(a) of the manuscript is given at time $t$, then the neighborhoods at time $t$ are fully determined (e.g., $\mathcal{N}_{1,t}=\{2,3,4\}$, $\mathcal{N}_{2,t}=\{1,3,4,9\}$, etc.); at a different time $t'$ if the topology is altered, then the neighborhoods $\mathcal{N}_{i,t'}$ might change too.




Letting
\[
x_i =(x_{1,i},x_{2,i}) = (\cos(\theta_i),\sin (\theta_i)),
\]
we can rewrite the Kuramoto model \eqref{kuramoto} as \eqref{eqn_nod} with $n_i=2$,
\BEA
\alpha_i = \begin{pmatrix} 0 & -\omega_i \\ \omega_i & 0\end{pmatrix},\quad
\beta_{ij} = \begin{pmatrix} \kappa_{ij} & 0 \\ 0 & \kappa_{ij} \end{pmatrix},
\notag
\EEA
$f_i(x_i) = x_i$, 
$f_{ij}(x_i,x_j) = \begin{pmatrix} -x_{2,i}(x_{2,j}x_{1,i}-x_{1,j}x_{2,i}) \\ x_{1,i}(x_{2,j}x_{1,i}-x_{1,j}x_{2,i})
\end{pmatrix}
$, and the summation is over 1-hop neighbors. 

\end{example}

\begin{example}[Power system]\label{eg:power_1}

The transient dynamics of a power system is governed mainly by DERs and power loads.  Conventionally the dynamics is written as a set of nonlinear differential-algebraic equations (DAEs), however, the dynamics can be converted to a node-wise form that resembles \eqref{eqn_nod}; see the detailed derivation in Appendix \ref{sec_pow}.

Consider a power system of $N$ nodes.  For a node `$i$' that is connected to DERs, the states are $x_i=[\vx_i,\vy_i]$, where $\vx_{i}$ are the state variables of the $i^{\text{th}}$ DER unit and $\vy_i$ are the algebraic variables, e.g., bus voltage amplitude and angle.  The dynamics of the DER node is,
\begin{equation}\label{eqn_pow_der}
    \begin{bmatrix}
        \dot{\vx}_i \\ \dot{\vy}_i
    \end{bmatrix} =
    \underbrace{\begin{bmatrix}
        \vK & O \\ O & \bar{\vY}^{-1}_{ii}
    \end{bmatrix}}_{\equiv \alpha_i}
    \underbrace{\begin{bmatrix}
        \vF_i(\vx_i,\vy_i) \\ \vG_i(\vx_i,\vy_i)
    \end{bmatrix}}_{\equiv f_i(x_i)}
    +
    \sum_{i\neq j}
    \underbrace{\begin{bmatrix}
        O & O \\ O & \bar{\vY}^{-1}_{ij}
    \end{bmatrix}}_{\equiv \beta_{ij}}
    \underbrace{\begin{bmatrix}
        O \\ \vG_j(\vx_j,\vy_j)
    \end{bmatrix}}_{\equiv f_{ij}(x_i,x_j)},
\end{equation}
where the definitions of each term are provided in Appendix \ref{sec_pow}.

For a non-DER node `$i$', the states reduce to $x_i=[\vy_i]$ and the dynamics is
\begin{equation}\label{eqn_pow_nde}
    \dot{\vy}_i = \sum_{i\neq j} \underbrace{\bar{\vY}^{-1}_{ij}}_{\equiv \beta_{ij}} \underbrace{\vG_j(\vx_j,\vy_j)}_{\equiv f_{ij}(x_i,x_j)},
\end{equation}
where $\beta_{ij}=0$ if $j\notin\cR$ and $\cR\subset\cV$ denotes the set of DER nodes; $\alpha_i=0$ due to the lack of the autonomous term.
The power system dynamics is highly heterogeneous due to two different types of nodal dynamics (DER nodes and non-DER nodes).  Furthermore, note that the summations in \eqref{eqn_pow_der} and \eqref{eqn_pow_nde} for a node involve nodes beyond the 1-hop neighbors.  This shows interactions between nodes that are not directly physically connected, and hence the system has $K$-hop interactions, $K>1$, which further complicates the modeling of power system dynamics.  The specific value of $K$ will be discussed in the Sec. \ref{sec5}.

\end{example}

\section{Reduced-Order Modeling Based on Mori-Zwanzig Formalism}\label{sec3}


As we stated in the introduction, our aim is to identify a reduced-order model to predict an observable $g:\mathbb{R}^{N_x} \to\mathbb{R}^n$ given time series of $u_i\equiv g(x(t_i))$ and the coarse-grained graph topology over a time sequence $\cT=\{t_1,t_2,\cdots,t_T\}$, evenly sampled with step size $\Dt$. Practically, we would like a reduced-order model to account for the knowledge of the graph structure that behaves as a time-dependent input parameter in the dynamical model. In \eqref{eqn_dyn}, while parameters $\beta_{ij}$ may not be available, the graph topology is encoded in the set $\cN_{i,t}^{(K)}$, which is assumed to be known.

One of the key issue here is that the dynamical system for arbitrary observable $g$ may not be written as a closed system of ODEs. To encompass this issue, we will make use of the Mori-Zwanzig formalism, which employs the Duhamel's principle to represent the dynamical system for the observable $g$ as a system of integro-differential equations (or non-Markovian).



For the affine control system in \eqref{eqn_dyn}, let us define the parameter maps, $\phi^t:\cA \times \cB \to \cA$, $t\in \bR^+$, where $\cB$ is the space of admissible control signals, $\{\beta_{ij}(\cdot)\}$, where $\beta_{ij}(\cdot):\bR^+ \to \bR^{n_i\times n_j}$. Consider a Banach space of observable,
\[
\cF = C(\cA,\bR^n)\equiv\{v:\cA\to\bR^n, v \text{ is continuous}\},
\]
then one can define a family of Koopman operators $U^t:\cF \to \cF$, 
\[
U^t g(x) = g \circ \phi^t(x,\{\beta_{ij}\}),
\]
for all $g\in \cF$. If $g\in C^1(\cA,\bR^n)$, then $u(x,t)\equiv U^tg(x)$ solves the PDE 
\BEA
\partial_t u(x,t) = \cL u(x,t), \quad\quad u(x,0)=g(x),\label{PDE}
\EEA
where
$\cL = (Af+B\otimes F)\cdot \nabla,$
denotes the infinitesimal generator of $U^t$ \cite{mauroy2020introduction,mauroy2016global,lasota2013chaos}. For the affine control system in \eqref{eqn_dyn}, $\cL u$ is bilinear in $u$ and $\beta_{ij}$ (see detailed discussion on the bilinearity in, e.g., Sections 1.2 and 4.1 of \cite{mauroy2020introduction}).

\subsection{Coarse-grained projection}\label{sec:CGG}
While the solution of the PDE in \eqref{PDE} characterizes the observable that is governed by the control affine problem in \eqref{eqn_dyn}, the fact that $u$ depends on the full state $x \in \bR^{N_x}$ is not computational practical. 
What we are often interested is to characterize the dynamics of an observable that depends on a lower dimensional coarse-grained (or resolved) variable. In the following, we define a generic coarse-graining strategy on a graph and the projection operators for such coarse-graining. Subsequently, we apply the projection to the dynamics and provide two examples.

\subsubsection{Coarse-graining of graph}\label{sec_cg}

Consider a generic coarse-graining strategy on graph as follows.  First, divide the nodes in $\cV$ into $M$ non-overlapping sets of nodes $\{\cV_i\}_{i=1}^M$, so that $\cV=\bigcup_{i=1}^M \cV_i$.
For each set $\cV_j$, we introduce the decomposition of resolved and unresolved variables, $\bar{x}_j\in\bR^{m_j}$ and $x_j'\in\bR^{s_j}$, such that
$$
x_i = \Phi_{ij}\bar{x}_j + \Psi_{ij}x_j',\quad i\in\cV_j,
$$
where $\Phi_{ij}\in\bR^{n_i\times m_j}$ and $\Psi_{ij}\in\bR^{n_i\times s_j}$.  
Here, we set $s_j = \sum_{i=1}^{|\cV_j|} n_i - m_j$, such that the decomposition recovers the space spanned by $\{x_i\}_{i\in \cV_j}$, whose dimension is $\sum_{i=1}^{|\cV_j|}n_i$. That is,
\[
\begin{pmatrix}\vdots \\ x_i \\ \vdots\end{pmatrix} = \begin{pmatrix}\vdots \\ \Phi_{ij} \\ \vdots\end{pmatrix} \bar{x}_j + \begin{pmatrix} \vdots  \\\Psi_{ij} \\ \vdots \end{pmatrix} x_j' \equiv {\underbrace{\Phi_j}_{\text{rank } m_j } \bar{x}_j + \underbrace{\Psi_j}_{\text{rank } s_j} x'_j}, \quad \forall i \in \cV_j.
\]
We require {$\Phi_{j}$} and $\Psi_{j}$ to be of full column rank, and the following orthogonality condition to hold 
\begin{equation}\label{eqn_orth}
    \Phi_j^\top\Psi_j = \sum_{i\in\cV_j} \Phi_{ij}^\top\Psi_{ij} = O.
\end{equation}

Up to this point, we have obtained a coarse-grained graph $\bar{\cG}$ of $M$ nodes $\bar{\cV}=\{i\}_{i=1}^M$, and each node has state $\bar{x}_i$.  The new edges are  $\bar{\cE}\subset \bar{\cV}\times \bar{\cV}$. Here, $(i,j)\in \bar{\cE}$ if there exists nodes 
$k\in \cV_i$ and $l\in\cV_j$ that are connected, $(k,l)\in \cE$. If $l \in \cN_{k,t}^{(K)} \cap \cV_j$, then $j \in \bar{\cN}_{i,t}^{(K)}$.

Collectively, denote the resolved states as 
$
\bar{x} =(\bar{x}_1,\ldots, \bar{x}_M)^\top\in \bR^n
$, where $n=\sum_{i=1}^M m_i$, and unresolved states as
$
x' =(x_1',\ldots, x_M')^\top\in \bR^{N_x-n}
$.
The state decomposition on graph is denoted
\begin{equation}
    x = \Phi\bar{x} + \Psi x',\notag
\end{equation}
where $\Phi\in\bR^{N_x\times n}$ is a $N\times M$ block matrix, 
$$
[\Phi]_{ij} = \begin{cases}\Phi_{ij}, & i\in \cV_j, \\ 0, & \text{otherwise,}
\end{cases}
$$
where $\Phi_{ij} \in \mathbb{R}^{n_i\times m_j}$. Also, $\Psi \in \mathbb{R}^{N_x\times (N_x-n)}$ has the same $N\times M$ block structure,
$$
[\Psi]_{ij} = \begin{cases}\Psi_{ij}, & i\in \cV_j, \\ 0, & \text{otherwise,}
\end{cases}
$$
where $\Psi_{ij}\in \mathbb{R}^{n_i \times s_j}$. With the choice of $s_j = \sum_{k=1}^{|\cV_j|}n_k - m_j$, it is clear that $\sum_{j=1}^M s_j = N_x-n$.
By the orthogonality condition \eqref{eqn_orth}, it is clear that $\Phi$ and $\Psi$ are orthogonal to each other, i.e., $\Phi^\top\Psi=O$.

\subsubsection{Definition of projection}
Next we define the projection operators for the above coarse-graining strategy on graph.

First, we introduce the left inverse of $\Phi$ as $\Phi^+\in\bR^{n\times N_x}$, such that $\Phi^+\Phi=I$ and $\Phi^+\Psi=O$; similarly define $\Psi^+\in\bR^{(N_x-n)\times N_x}$ such that $\Psi^+\Psi=I$ and $\Psi^+\Phi=O$.  For convenience, we define the projection matrices $P \equiv \Phi\Phi^{+}$ and $Q \equiv I - P = \Psi\Psi^{+}$. Since $|\Phi|_2 = S <\infty$ is usually given, where {$|\cdot|_2$} denotes the usual matrix 2-norm, it will be convenient to choose $\Psi$ such that {$|\Psi|_2 = R$} to simplify the scaling law calculation in Section~\ref{sec:memory}.
In the following derivation, we choose $\Phi^+$ to be a $M\times N$ block matrix, where $[\Phi^+]_{ij}=0$ if $j\notin\cV_i$; similarly for $\Psi^+$.
In Appendix \ref{sec_lin_inv}, we show the construction of such matrices.

Next, for any continuous real-valued function $g:\cA\in \bR^{N_x} \to \bR$, we define the projection operator as,
\begin{equation}
    \cP g(x) = g(Px) \equiv \bar{g}(\bar{x}),\notag
\end{equation}
restricting the function to depends only on $\bar{x}$. Correspondingly, the residual
\begin{equation}
    \cQ g(x) = g(x) - \bar{g}(\bar{x}) \equiv g'(x).\notag
\end{equation}
Clearly $\cP \bar{g}(\bar{x})=\bar{g}(\bar{x})$ and $\cP g'(x)=0$. We can extend $\cP$ to any multi-variable functions by employing the projection on each component. Relevant to us is to employ the projection operator to functions in $\cH = C(\cA,\bR^n)$ as well as in $C(\cA,\bR^{N_x-n})$. Since our intention is to express the equation to reduced-order coordinates, we require
\begin{equation}\label{eqn_prj_p0}
    \cP (M g(x))=MPg(Px),
\end{equation}
for any $M\in \bR^{m\times N_x}$, $g:\bR^{N_x} \to \bR^{N_x}$, and $m\in \mathbb{N}$.  For the ease of subsequent derivations, \eqref{eqn_prj_p0} is further manipulated as follows,
\begin{align}\label{eqn_prj_p}
    \cP (M g(x)) &= MPg(Px) = (M\Phi) \Phi^+ \bar{g}(\bar{x}) \equiv (M\Phi) \overline{g^{\Phi}}(\bar{x}).
\end{align}
Similarly, the residual is decomposed as follows,
\begin{align}\nonumber
    \cQ (M g(x)) &= M g(x) - \cP (M g(x)) = M (P+Q) g(x) - (M\Phi) \overline{g^{\Phi}}(\bar{x}) \\\nonumber
    &= (M\Phi) \Phi^+ (g(x) - \bar{g}(\bar{x})) + (M\Psi) \Psi^+ g(x) \\\label{eqn_prj_q}
    &\equiv (M\Phi) \widetilde{g^\Phi}(x) + (M\Psi) g^\Psi(x).
\end{align}
In sum, through the projection we have defined $\overline{g^{\Phi}}(\bar{x}) \equiv \Phi^+ \bar{g}(\bar{x})$, $\widetilde{g^\Phi}(x) \equiv \Phi^+ (g(x) - \bar{g}(\bar{x}))$, and $g^\Psi(x) = \Psi^+ g(x)$.

\subsubsection{Projection of dynamics}

Let us write the right hand side of \eqref{eqn_dyn} as follows,
\BEA
\dot{\bar{x}} &=& \Phi^+ \dot{x} = h_1(x) + H_{1}(x,x), \notag \\
\dot{x}' &=& \Psi^+ \dot{x} = h_2(x) + H_{2}(x,x), \notag
\EEA
where
\BEA
h_1(x) &=& \Phi^+(Af(x))\quad
H_{1}(x,x) = \Phi^+(B(t)\otimes F(x,x)), \notag \\
h_2(x) &=& \Psi^+(Af(x)), \quad
H_{2}(x,x) = \Psi^+(B(t)\otimes F(x,x)). \notag
\EEA
We can rewrite the dynamics as,
\begin{subequations}\label{eqn_dec}
\begin{align}
    \dot{\bar{x}} &= \underbrace{\cP\left( h_1(x)+ H_{1}(x,x)\right)}_{\equiv r^{(1)}(\bar{x})} + \underbrace{\cQ\left(h_1(x)+ H_{1}(x,x)\right)}_{\equiv r^{(2)}(x)} \\
    \dot{x}' &= \underbrace{\cP\left( h_2(x)+ H_{2}(x,x)\right)}_{\equiv r^{(3)}(\bar{x})} + \underbrace{\cQ\left(h_2(x)+ H_{2}(x,x)\right)}_{\equiv r^{(4)}(x)}. 
\end{align}
\end{subequations}
Next the terms $r^{(1)}$ and $r^{(2)}$ are computed, as they are needed for the subsequent analysis.  First term of $r^{(1)}$ is treated (see Appendix \ref{sec_lin_dia} for details). Using \eqref{eqn_prj_p},
$$
\cP h_1(x) = \cP(\Phi^+ A f(x)) = \Phi^+ A\Phi \overline{f^\Phi}(\bar{x}) \equiv A^{11} \overline{f^\Phi}(\bar{x}),
$$
and the components are denoted
\BEA
[\cP h_1(x)]_i = \left[A^{11} \overline{f^\Phi}(\bar{x})\right]_i \equiv \alpha_i^{11} \overline{f^\Phi_i}(\bar{x}_i).\label{proj_g2}
\EEA
{In the above, we have defined,
\[\alpha_i^{11} = [A^{11}]_{ii} = [\Phi^+A\Phi]_{ii}.\]}

Similarly, one obtains the expression for $\cP H_{1}(x,x)$ (see Appendix \ref{sec_lin_oti} for details).  Its $i$th component is
\begin{equation}\label{eqn_ph1_0}
[\cP H_{1}(x,x)]_i = \sum_{j \in \bar{\cN}_{i,t}^{(K)}} \sum_{k\in\cV_i}\beta^{11}_{ijk} \overline{f^\Phi_{kj}}(\bar{x}_i,\bar{x}_j) \in \bR^{m_i}, \quad i = 1,\ldots M,
\end{equation}
where $\bar{\cN}_{i,t}^{(K)}$ denotes the set of $K$-hop neighbors of the $i$th coarse-grained node, and
$$
\beta^{11}_{ijk} = \sum_{\ell\in \cV_j}\Phi_{ik}^+ \beta_{k\ell} \Phi_{\ell j},\quad \overline{f^\Phi_{kj}}(\bar{x}_i,\bar{x}_j) = \sum_{\ell\in \cV_j} \Phi_{j\ell}^+ {\bar{f}}_{k\ell}(\bar{x}_i,\bar{x}_j),
$$
for all $k\in\cV_i$.
\begin{rem}\label{rem_hom}
    If the interaction terms are homogeneous within $\cV_i$, i.e.,  ${\bar{f}}_{pj}(\bar{x}_i,\bar{x}_j)={\bar{f}}_{qj}(\bar{x}_i,\bar{x}_j)$ for $\forall p,q\in\cV_i$, then \eqref{eqn_ph1_0} simplifies as,
    \begin{equation}\label{eqn_ph1}
    [\cP H_{1}(x,x)]_i = \sum_{j \in \bar{\cN}_{i,t}^{(K)}} \beta^{11}_{ij} \overline{f^\Phi_{ij}}(\bar{x}_i,\bar{x}_j),
    \end{equation}
    where $\overline{f^\Phi_{ij}}(\bar{x}_i,\bar{x}_j)=\overline{f^\Phi_{kj}}(\bar{x}_i,\bar{x}_j)$ for $\forall k \in\cV_i$, and
    $$
    \beta^{11}_{ij} = \sum_{k\in \cV_i}\beta^{11}_{ijk} = \sum_{k\in \cV_i}\sum_{\ell\in \cV_j}\Phi_{ik}^+ \beta_{k\ell} \Phi_{\ell j}.
    $$
    Using $\otimes$ notation, one can write \eqref{eqn_ph1} compactly as
    $$
    \cP H_{1}(x,x) = B^{11}(t) \otimes \overline{F^\Phi}(\bar{x},\bar{x}),
    $$
    where $[B^{11}(t)]_{ij} = [\Phi^+ B(t) \Phi]_{ij} = \beta^{11}_{ij}$ and {$[\overline{F^\Phi}(\bar{x},\bar{x})]_{ij} = \overline{f_{ij}^\Phi}(\bar{x}_i,\bar{x}_j)$.}

    Note that after coarse-graining new self-interaction terms with strength $\beta_{ii}^{11}$  may emerge.
\end{rem}

The computation of $r^{(2)}$ is similar to $r^{(1)}$, except that \eqref{eqn_prj_q} is used in place of \eqref{eqn_prj_p}.  The only difference is that the residual operator $\cQ$ produces two terms for each of $\cQ h_1 (x)$ and $\cQ H_{1} (x,x)$.  Specifically, the components of the residual are
\begin{subequations}\label{proj_g34}
\begin{align}
[\cQ h_1 (x)]_i &= \alpha^{11}_i \widetilde{f^\Phi_i}(x_i)+  \alpha^{12}_{i}f^\Psi_i(x_i),\label{proj_g3} \\
[\cQ H_{1} (x,x)]_i &= \sum_{j\in\bar{\cN}^{(K)}_{i,t}} \sum_{k\in \cV_i} \left(\beta_{ijk}^{11}  \widetilde{f^\Phi_{kj}}(x_k,\bar{x}_i,x_{\cV_j}) + \beta_{ijk}^{12} f^\Psi_{kj}(x_k,x_{\cV_j}) \right), \label{proj_g4}
\end{align}
\end{subequations}
where we have denoted
$$
\alpha_i^{12} =[\Phi^+A\Psi]_i,\quad \beta_{ijk}^{12} = \sum_{\ell\in \cV_j} \Phi_{ik}^+ \beta_{k\ell}\Psi_{\ell j},
$$
and
\begin{align*}
\widetilde{f_{kj}^\Phi} 
(x_k,\bar{x}_i,x_{\cV_j}) &=  \sum_{\ell\in \cV_j}\Phi_{j\ell}^+ (f_{k\ell}(x_k,x_\ell)) - \bar{f}_{k\ell}(\bar{x}_i,\bar{x}_j)),\notag  \\
f^\Psi_{kj}(x_k,x_{\cV_j}) &= \sum_{\ell\in \cV_j} \Psi^+_{j\ell}f_{k\ell}(x_k,x_\ell),\notag
\end{align*}
for all $k\in\cV_i$, with $x_{\cV_j} = \{x_\ell: \ell \in \cV_j\}$. See Appendix~\ref{sec_lin_oti} for the detailed derivation of these interaction terms.
Following the same argument as above, one can deduce $r^{(3)}$ and $r^{(4)}$.
In the new coordinates $x =\Phi\bar{x}+ \Psi x'$. By chain rule, we have,
$
\nabla_x = (\Phi^+)^\top \nabla_{\bar{x}}  + (\Psi^+)^\top \nabla_{x'},
$ and subsequently, the infinitesimal generator of the Koopman operator is given as,
\begin{eqnarray}
\mathcal{L} &=& \dot{x} \cdot \nabla_x = \dot{x}^\top ((\Phi^+)^\top \nabla_{\bar{x}}  + (\Psi^+)^\top \nabla_{x'}) \notag \\ &=& (\Phi^+\dot{x})^\top\nabla_{\bar{x}} + (\Psi^+\dot{x})^\top\nabla_{x'} 
= \dot{\bar{x}}\cdot\nabla_{\bar{x}} + \dot{x}' \cdot\nabla_{x'} \notag \\ &=&  (r^{(1)}(\bar{x})+r^{(2)}(x))\cdot  \nabla_{\bar{x}} + (r^{(3)}(\bar{x})+r^{(4)}(x))\cdot \nabla_{x'}.
\label{LE}
\end{eqnarray}

This change of coordinate ensures the following identities,
\BEA
    \cP r^{(i)}(\bar{x}) &=& r^{(i)}(\bar{x}),\quad i=1,3 \quad \text{ and } \quad \cP r^{(j)}(x) = 0, \quad j=2,4,\label{projid}
\EEA
to ease the computations in the rest of this paper. It is important to note that $r^{(1)}$ and $r^{(2)}$ depend linearly on 
the topology of the coarse-grained or resolved variables, $B^{11}(t) =(\beta^{11}_{ij}(t))$. To emphasize this linear parameter dependence, we will write $r^{(i)} \equiv r^{(i)}(\bar{x};B^{11}(t))$ for $i=1,2$. 

\subsubsection{Examples}

\begin{example}[Kuramoto example, continued]\label{eg:kuramoto_2}
Recall that in this example, nodal state dimension is $n_i=2$ and total state dimension is $N_x = \sum_{i=1}^N n_i = 2N$.
Following the coarse-graining strategy in Sec. \ref{sec_cg}, we choose to divide the nodes as $\cV=\bigcup_{i=1}^M\cV_i$.  Without loss of generality, let the first $|\cV_1|$ nodes belong to $\cV_1$, the next $|\cV_2|$ nodes belong to $\cV_2$, etc. For the specific topology represented in Figure~\ref{kuramoto_topo}, we have divided every 4 consecutive nodes into one group, so $\cV_1=\{1,2,3,4\}$, $\cV_2=\{5,6,7,8\}$, $\cdots$, $\cV_5=\{17,18,19,20\}$.  Each group produces one node after the coarse-graining, resulting in a total of $M=5$ nodes.
Next, for group $\cV_i$ define an averaged state
\begin{equation}\label{eqn:obs}
    \bar{x}_i := \frac{1}{2|\cV_i|} \sum_{j \in \cV_i} (x_{1,j}+x_{2,j}) \in \bR^{m_i}, \quad \quad i=1,\ldots, M.
\end{equation}
The observable of interest is
$$
\bar{x} = (\bar{x}_1, \ldots, \bar{x}_M) \in \bR^{n}.
$$
Following earlier notations, we have chosen $m_i = 1$ as the dimension of the resolved state of one node and $n=\sum_{i=1}^M m_i = M$ as the total dimension of the resolved states.  A particular instance of coarse-graining is provided in the numerical results section, in which a specific graph topology is given.

The projection matrices associated with $\bar{x}$ are given by,
\[
\Phi = \begin{pmatrix} \Phi_1 \\ & \Phi_2 \\ & & \ddots \\ & & & \Phi_M \end{pmatrix},\ \Phi^+ = \begin{pmatrix} \Phi_1^+ \\ & \Phi_2^+ \\ & & \ddots \\ & & & \Phi_M^+ \end{pmatrix},
\]
where $\Phi \in \bR^{N_x\times n}$, $\Phi^+ \in \bR^{n\times N_x}$, and
\[
\Phi_i^\top = \begin{pmatrix}1 & 1 & \ldots & 1 \end{pmatrix}\in\bR^{2|\cV_i|},\ \Phi_i^+=\frac{1}{2|\cV_i|}\Phi_i^\top.
\]
Here, we note that $ R\equiv |\Phi|_2 = \max_i \sqrt{2|\cV_i|}$.
Subsequently, since the interaction terms $f_{ij}$ in this example are homogeneous by the definition in Remark \ref{rem_hom}, it is sufficient to compute $B^{11}$ after coarse-graining,
\begin{equation}\label{eqn:adj}
    B^{11} = \Phi^+\begin{pmatrix} 0 & \beta_{12} & \ldots & \beta_{1N} \\
    \beta_{12} & 0 & \ldots & \beta_{2N} \\
    \vdots &  & \ddots & \vdots \\ \beta_{N1} & \beta_{N2} & \ldots & 0
    \end{pmatrix}\Phi \equiv \begin{pmatrix} \beta^{11}_{11} & \beta^{11}_{12} & \ldots & \beta^{11}_{1M} \\
    \beta^{11}_{12} & \beta^{11}_{22} & \ldots & \beta^{11}_{2M} \\
    \vdots &  & \ddots & \vdots \\ \beta^{11}_{M1} & \beta^{11}_{M2} & \ldots & \beta^{11}_{MM}
    \end{pmatrix},
\end{equation}
where
$$
\beta_{ij}^{11} =\frac{1}{2|\cV_i|}\sum_{k\in \cV_i,\ell\in  \cV_{j}}  \kappa_{k\ell},\quad (k,\ell)\in\bar{\cE}.
$$
Via a similar procedure, one can find
$$
\alpha_{i}^{11} = 0,\quad \beta_{ii}^{11} = \frac{1}{2|\cV_i|} \sum_{k,\ell\in\cV_i} \kappa_{k\ell}.
$$
Above results indicate that under the chosen projection, the self-interaction in the coarse-grained Kuramoto system is only due to the interactions between the nodes within the same group, instead of the autonomous terms driven by $\alpha_i$.
\end{example}

\begin{example}[Power system example, continued]\label{eg:power_2}
In this example, we are interested to identify the dynamics of the observable $\vy_i$ in both the DER and non-DER nodes. Since the observable does not require a coarse-graining (or averaging) over different nodes, the projection operator is simplified, that is $\cP f(\vx,\vy) = f(0,\vy)$, for any function $f$. Since $\Phi$ and $\Psi$ are identities on the $\vy$ components and their orthogonal directions, correspondingly, we have
\[
{\alpha^{11}_i = \bar{\vY}^{-1}_{ii}}
\]
for the DER node and zero for the non-DER node. Correspondingly,
\[
{\beta^{11}_{ij} = \bar{\vY}^{-1}_{ij}}
\]
for both the DER and non-DER nodes.
\end{example}


\subsection{Mori-Zwanzig Expansion}


Suppose that $g(x) = \bar{x}$, then $u(x,t) = e^{t\cL} g(x) = e^{t\cL} \bar{x} $, so the PDE in \eqref{PDE} can be equivalently written as,
\begin{equation}
    \ppf{}{t}e^{t\cL} \bar{x} = \cL e^{t\cL} \bar{x}.\notag
\end{equation}
We should clarify that the initial condition of the PDE in \eqref{PDE} corresponds to the system of ODEs in \eqref{eqn_dec}, with initial condition $x(0) = x$.

Using the Dyson's formula, one can write the Mori-Zwanzig (MZ) expansion \cite{mori:65,zwanzig:73,chk:00},
\begin{equation}
\ppf{}{t}e^{t\cL} \bar{x} =  \underbrace{e^{t\cL}\cP\cL \bar{x}}_{\text{Markovian}} + \underbrace{\int_{0}^t e^{(t-s)\cL} \cP\cL e^{s\cQ\cL} \cQ\cL \bar{x} ds}_{\text{non-Markovian}}+ \underbrace{e^{t\cL}\cQ\cL\bar{x}, }_{\text{orthogonal dynamics}}\notag
\end{equation}
where $\cQ = \cI-\cP$. Projecting to the range of $\cP$, we obtain,
\begin{equation}\label{eqn:dyn}
    \ddt{\bar{x}(t)} = \ppf{}{t}\cP e^{t\cL} \bar{x} =  \underbrace{\cP e^{t\cL}\cP\cL \bar{x}}_{\text{Markovian}} + \underbrace{\int_{0}^t \cP e^{(t-s)\cL} \cP\cL e^{s\cQ\cL} \cQ\cL \bar{x} ds}_{\text{non-Markovian}}.
\end{equation}
Using \eqref{projid}, the Markovian term is simply,
\begin{align}
\cP e^{t\cL}\cP\cL \bar{x} &= \cP e^{t\cL} \cP r(x)\cdot \nabla_{x} \bar{x}  \notag \\
&= \cP e^{t\cL} \cP (r^{(1)}(\bar{x})+ r^{(2)}(x) )\cdot \nabla_{\bar{x}} \bar{x}  \notag \\
&= \cP e^{t\cL}r^{(1)}(\bar{x}) = r^{(1)}(\cP e^{t\cL}\bar{x}) = r^{(1)}(\bar{x}(t)).\notag
\end{align}
The non-Markovian term depends on the solution of the orthogonal dynamics,
\begin{equation}
w(x,t) = e^{t\cQ\cL} \cQ\cL\bar{x} = e^{t\cQ\cL} r^{(2)}(x).\notag
\end{equation}

One of the issues in leveraging the MZ expansion is that  solving the integro-differential equation in \eqref{eqn:dyn} is computationally more expensive than solving the full ODE system in \eqref{eqn_dec}. This issue is mainly due to the
lack of explicit expression for the solution operator $e^{t\cQ\cL}$ of the orthogonal dynamics.

Since our goal is to design a machine learning algorithm, or more specifically, to determine a graph neural network architecture that encodes the graph topology, we would like to verify how the non-Markovian term depends on the parameter $B^{11}$ that appears in $r^{(1)}$ and $r^{(2)}$. While this is a difficult task since we don't have an explicit solution for the orthogonal dynamics, we can obtain an approximation of how the non-Markovian term depends on these parameters by inspecting the leading order terms in the Taylor's expansion of the orthogonal dynamics for small $t>0$, 
\begin{equation}
e^{t\cQ\cL} = I + t\cQ\cL + O(t^2).
\end{equation}
While it is unclear whether such an expansion is valid, the leading order terms have been considered for approximating the orthogonal dynamics \cite{chk:00,stinis2007higher}. As we stated above, we will only inspect how the memory terms depend on the parameters $B^{11}$ in the leading order terms. Mathematically, the order-$s$ term in this expansion operates on an observable that is also a function of $C^{s+1}(\cA,\bR^n)$. 

Based on this expansion, the leading order non-Markovian term is given as,
\BEA
\begin{aligned}
\int_{0}^t \cP e^{(t-s)\cL} \cP\cL e^{s\cQ\cL} \cQ\cL \bar{x} ds &= \underbrace{\int_{0}^t \cP e^{(t-s)\cL} \cP\cL \cQ\cL \bar{x} ds }_{=I_1}\\\label{eqn_mz_exp}
&+ \underbrace{\int_{0}^t s\cP e^{(t-s)\cL} \cP\cL \cQ\cL\cQ\cL \bar{x} ds}_{= I_2} + O(t^3).
\end{aligned}
\EEA
Since $\cQ\cL\bar{x} = r^{(2)}(x)$, 
\begin{align}
\cP\cL\cQ\cL\bar{x} &= \cP\cL r^{(2)}(x) = \cP \left((r^{(1)}+r^{(2)})\cdot \nabla_{\bar{x}}r^{(2)}(x) + (r^{(3)}+r^{(4)})\cdot \nabla_{x'} r^{(2)}(x)\right) \notag \\
&= \cP\left((r^{(1)}+r^{(2)})^\top  \nabla_{\bar{x}}r^{(2)}+ (r^{(3)}+r^{(4)})^\top \nabla_{x'} r^{(2)}(\bar{x};B^{11})\right), \notag \\
&= (r^{(1)}(Px)+\underbrace{r^{(2)}(Px)}_{=0})^\top \nabla_{\bar{x}}r^{(2)}(Px)+ \underbrace{(r^{(3)}+r^{(4)})^\top \nabla_{x'} r^{(2)}(\bar{x};B^{11})}_{=0}, \notag \\
&= r^{(1)}(\bar{x})\cdot  \nabla_{\bar{x}}r^{(2)}(\bar{x}).\notag
\end{align}
which has a quadratic dependence on $B^{11}$.










Since $B^{11}$ implicitly encodes the topology of the pair-wise interaction between coarse-grained variables among the $K$-hop neighbors, the quadratic dependence on $B^{11}$ in the leading order expansion of the non-Markovian term motivates the use of a class of model that allows for pair-wise interactions of $2K$-hop neighbors.

\subsection{Memory length}\label{sec:memory}

In this section, our goal is to quantify the required memory length when using a model that accounts for the interaction of $2K$-hop neighbors. Particularly, we will deduce a minimum memory length based on a scaling argument of the leading-order of the MZ memory term, $I_1$.

To facilitate the analysis below, we should clarify that the norm $\|\cdot\|$ is associated to the Banach space $C(\cA,\bR^m)$ for any $m\in \mathbb{N}$. The leading order non-Markovian term is,
\[
I_1 = \int_0^t r^{(1)}(\bar{x}(t-s)) \cdot  \nabla_{\bar{x}}r^{(2)}(\bar{x}(t-s)) \,ds
\]
where
\begin{align}
r^{(1)}_i(\bar{x}) &= \alpha_i^{11} \overline{f^\Phi_i}(\bar{x}_i) + \sum_{j \in \bar{\cN}_{i,t}^{(K)}} {\sum_{k\in \cV_i} \beta^{11}_{ijk} \overline{f^\Phi_{kj}}}(\bar{x}_i,\bar{x}_j) \notag \\
r^{(2)}_i(x) &= \alpha^{11}_i \widetilde{f^\Phi_i}(x_i)+  \alpha^{12}_{i}f^\Psi_i(x_i) \notag \\ &\quad + \sum_{j\in\bar{\cN}^{(K)}_{i,t}} {\sum_{k\in \cV_i} \left(\beta_{ijk}^{11}  \widetilde{f^\Phi_{kj}}(x_k,\bar{x}_i,x_{\cV_j}) + \beta_{ijk}^{12} f^\Psi_{kj}(x_k,x_{\cV_j}) \right).} \notag 
\end{align}
In the derivation below, we will use the matrix max norm to denote $|A| = \max_{ij} |a_{ij}|$ for any matrix $A$ and $|\cdot|_2$ as the usual matrix 2-norm. Since the vector fields $f_i,f_{ij} \in C^1(\cA)$, and, $|\Phi|\leq |\Phi|_2 { =R<\infty, |\Psi|\leq |\Psi|_2 =R<\infty}$, we
have $\|\overline{f^\Phi}\|=\|\cP( f^\Phi)\| = |{\Phi^+}|\|\cP f\| \leq {R^{-1}}\|f\| < \infty$. With similar argument, $\|\widetilde{f^\Phi}\|,\|f^\Psi\| \leq \infty$ {and their upper bounds are of order-$R^{-1}$. For ease of discussion below, define the Lipschitz constant,
\[
R^{-1}L\equiv \max\left\{\|\nabla f_{i}\|,\|\nabla F_{ij}\|\right\},
\]
where $f_i\in \{\overline{f^\Phi_i}, \widetilde{f^\Phi_i}, f^\Psi_i\}$, $F_{ij} \in \{\overline{f^\Phi_{ij}}, \widetilde{f^\Phi_{ij}},f^\Psi_{ij}\}$.}

\begin{prop}\label{prop_tde} Let $|\alpha_i| = O(1)$ and $|\beta_{ij}| \in O(\epsilon^p)$ for $j\in \bar{\cN}_{i,t}^{[p]}$. 
If $d\epsilon> 1/2$, then
\[
|I_1| = O \left(\frac{LT\left( d\epsilon\right)^{2} }{R^2(1-d\epsilon)^2}\right).
\]
\end{prop}

\begin{proof}
Let $a \in \{\alpha^{11}, \alpha^{12}\}$ and $b\in\{\beta^{11}, \beta^{12}\}$. 
We first note that, for $j\in \bar{\cN}_{i,t}^{[p]}$,
$$
|b_{ijk}|= \begin{cases}
\left|\sum_{\ell\in\cV_j} \Phi_{ik}^+ \beta_{k\ell} \Phi_{\ell j}\right| =O(\epsilon^p), & \text{ if } b_{ijk} = \beta_{ijk}^{11},  \\[8pt]
\left|\sum_{\ell\in\cV_j} \Phi_{ik}^+ \beta_{k\ell} \Psi_{\ell j}\right| =O(\epsilon^p), & \text{ if } b_{ijk} = \beta_{ijk}^{12}.
\end{cases}
$$
Hence for $\forall j\in \bar{\cN}_{i,t}^{[p]}$
\[
\left|\sum_{k\in \cV_i}b_{ijk}F_{kj}\right| \leq
\|F_{\cV_i,j}\| \sum_{k\in \cV_i} |b_{ijk}|  = O(\epsilon^pR^{-1}),
\]
where $\|F_{\cV_i,j}\| = \max_{k\in \cV_i}|F_{kj}|$.
With similar arguments, one can deduce that,
\[
\left|\sum_{k\in \cV_i}b_{ijk}\partial_{\bar{x}_i}F_{kj}\right| \leq \|F_{\cV_i,j}\| \sum_{k\in \cV_i} |b_{ijk}| = O(\epsilon^p R^{-1}L).
\]

Thus,
\BEA
s_{11}&:=&\left\|a_i f_i a_i \partial_{\bar{x}_i} f_i\right\| = O({R^{-2}}L), \notag\\
s_{12}(p)&:=& \left\|a_i f_i {\left(\sum_\kappa b_{ij\kappa} \partial_{\bar{x}_i} F_{\kappa j}\right)}\right\|  =O(\epsilon^{p}{R^{-2}}L), \quad \forall j \in \bar{\cN}_{i,t}^{[p]}, \notag \\
s_{21}(p)&:=&
\left\| {\left(\sum_\kappa b_{ij\kappa} F_{\kappa j}\right)} a_{i} \partial_{\bar{x}_i} f_i\right\|  = O(\epsilon^{p}{R^{-2}}L), \quad \forall j \in \bar{\cN}_{i,t}^{[p]}, \notag\\
s_{22}(p,q)&:=&\left\|{\left(\sum_\kappa b_{ij\kappa} F_{\kappa j}\right)} {\left(\sum_\kappa b_{i\ell \kappa} \partial_{\bar{x}_i} F_{\kappa\ell}\right)} \right\|  = O(\epsilon^{p+q}{S^{-2}}L),  
\notag \\ 
&& \hspace{1.5in}\forall j \in \bar{\cN}_{i,t}^{[p]}, \forall \ell \in \bar{\cN}_{i,t}^{[q]},\notag
\EEA
{where the terms $s_{12}(0)$, $s_{21}(0)$, and $s_{22}(0,0)$ represent the strengths of new self-interactions due to the memory integral.} Since 
\begin{eqnarray}
\sum_{j \in \bar{\cN}_{i,t}^{(K)}}{ \left(\sum_\kappa b_{ij\kappa} F_{\kappa j}\right)} &=& \sum_{{k=0}}^K\sum_{j \in \bar{\cN}_{i,t}^{[k]}} {\left(\sum_\kappa b_{ij\kappa} F_{\kappa j}\right)}, \notag 
\end{eqnarray}
and $|\bar{\cN}_{i,t}^{[k]}| = O((\frac{Md}{N})^k)$, where $M/N<1$, 
we have
\begin{align}
\|r^{(1)} \cdot \nabla_{\bar{x}} r^{(2)} \| &\leq  s_{11} + {s_{12}(0) + s_{21}(0) + s_{22}(0,0)} \notag \\
&\quad +  C_1 \sum_{k=1}^K \left(\frac{Md}{N}\right)^k (s_{12}(k)+s_{21}(k) )\notag \\
&\quad + C_2 \sum_{p,q=1}^K \left(\frac{Md}{N}\right)^{p+q} s_{22}(p,q) \notag \\
&= O({R^{-2}}L) + O \left({R^{-2}}L\sum_{k=1}^K\left(d\epsilon\right)^k\right) + O\left({R^{-2}}L\sum_{p,q=1}^K \left(d\epsilon\right)^{p+q}\right) \notag \\
&= O({R^{-2}}L)+O \left({R^{-2}}L(d\epsilon)\frac{1-(d\epsilon)^{K-1}}{1-d\epsilon} \right) \notag \\ &\quad\quad + O\left({R^{-2}}L
\frac{(d\epsilon)^2(1-(d\epsilon)^K)^2}{(1-d\epsilon)^2}\right)\notag \\
&= O({R^{-2}}L) +  O\left({R^{-2}}L\frac{d\epsilon}{1-d\epsilon}\right) + O\left({R^{-2}}L\frac{ (d\epsilon)^2}{(1-d\epsilon)^2}\right).\notag
\end{align}
Thus, it is clear that the third term dominates if $d\epsilon>1/2$, and the proof is complete after integrating this bound.
\end{proof}

\begin{rem}\label{rem_de}
If $d\epsilon> 1/2$, the proposition above suggests that if one considers a $2K$-hop model, then the leading
order memory term in the MZ equation 
is appropriately accounted with time lag $T\sim \frac{{R^2}(1-d\epsilon)^2}{L(d\epsilon)^2}$. On the other hand, if $d\epsilon \ll 1/2$, the leading order memory term need a longer time lag $T\sim {R^2}L^{-1}$ to be appropriately accounted. {In the Kuramoto example, since $R^2=2\max_{i=1,\ldots,M}|\cV_i|$, it is intuitive that longer memory is needed when each coarse-grained node is formed by aggregating a larger number of nodes.}
\end{rem}

\section{Graph Neural Network}\label{sec:gnn}

Informed by the MZ expansion, a graph neural network architecture is developed to implement the reduced-order model \eqref{eqn:dyn} based on the coarsened graph topology.

\paragraph{Strategy of GNN Modeling.}
The MZ calculation reveals that the rate of change in $\bx$ is approximately a function of the past $T$ measurements and graph topologies, plus the coarse-grained topology at the next step.  We use a GNN\CORR{, denoted by $F_G$,} to learn this relation,
\begin{equation}\label{eqn_gnn}
    \dot{\bx} = F_G(\bx_t,\bx_{t-1},\cdots,\bx_{t-T+1}; \bar{\cG}_t, \bar{\cG}_{t-1}, \cdots, \bar{\cG}_{t-T+1}; \bar{\cG}_{t+1};\vtQ),
\end{equation}
where \CORR{$\vtQ$ represents the learnable parameters; the detailed form of $F_G$ is provided near the end of this section}.  Subsequently the next partial measurement is predicted as
\begin{equation*}
    \bx_{t+1} = \bx_t + \Dt \dot{\bx},
\end{equation*}
where $\Dt$ is the time step size.

In the following, we first present the building blocks of the proposed GNN model, and then its overall architecture.

\paragraph{Weighted Graph Laplacian.}
To incorporate the interaction matrices $\alpha_i^{11}$ and $\beta_{ij}^{11}$ into the GNN model, the graph is augmented with edge weights, i.e., each edge $(i,j)$ is assigned a real value $\bar{w}_{ij}$.  The collection of weights is denoted by $\bar{\cW}=\{\bar{w}_{ij}|(i,j)\in\bar{\cE}\}$; the graph with edge weights is denoted by $\bar{\cG}=(\bar{\cV},\bar{\cE},\bar{\cW})$.  The computation of $w_{ij}$ from $\alpha_i^{11}$ and $\beta_{ij}^{11}$ is problem-dependent, and two examples are given below.

\begin{example}[Kuramoto example, continued]\label{eg:kuramoto_3}
Since $\alpha_i^{11}$ and $\beta_{ij}^{11}$ are already scalars, we simply choose
\CORR{
\begin{equation}\label{eqn_kura_wij}
\bar{w}_{ij} = \left\{\begin{array}{ll}
    \beta_{ii}^{11}, & i=j, \\
    \beta_{ij}^{11}, & i\neq j,\ j\in\bar{\cN}_{i,t}, \\
    0, & \text{Otherwise.}
\end{array}
\right.
\end{equation}
}\end{example}

\begin{example}[Power system example, continued]\label{eg:power_3}
The coefficient matrices $\alpha_i^{11}$ and $\beta_{ij}^{11}$ are proportional to the blocks of $\bar{\vY}^{-1}$.  However, in power system modeling, only $\bar{\vY}$ is available and it is inconvenient to obtain $\bar{\vY}^{-1}$.  To emulate the decreasing trend of $\bar{\vY}^{-1}$ using $\bar{\vY}$, the weights are defined as
\begin{equation}\label{eqn_pow_wij}
\bar{w}_{ij} = \exp(-\eta|\bar{Y}_{ij}|^{2}),
\end{equation}
where $\eta$ is a user-specified parameter that tunes the range of $w_{ij}$.
\end{example}

Furthermore, as needed by the GNN model, we introduce the \CORR{weighted graph Laplacian
\begin{equation}
    \vL=\vD-\vW,
\end{equation}}
where $\vW$ is the weight matrix
\begin{equation}
    [\vW]_{ij} = \left\{\begin{array}{ll}
        \bar{w}_{ij}, & (i,j)\in\bar{\cE}, \\
        0, & \text{otherwise,}
    \end{array}
    \right. \label{Graphweight}
\end{equation}
and $\vD$ is a diagonal matrix with $[\vD]_{ii}=\sum_{j\in\bar{\cN}_i}\bar{w}_{ij}$.

\paragraph{Message Passing (MP) Mechanism.}
The MP mechanism is the cornerstone for many GNN architectures, which consists of multiple consecutive MP steps \cite{Hamilton2020}.
Consider a graph $\bar{\cG}=(\bar{\cV},\bar{\cE},\bar{\cW})$ of $M$ nodes and a vector of features $h_v^{(0)}\in\bR^{D_0}$ at each node $v\in\bar{\cV}$, one MP step is defined as
\begin{subequations}
\begin{align}\label{eqn:ebd}
  &\text{Aggregate:}\quad m_v^{(0)} = M_{ag}\left(\{h_u^{(0)}\ |\ u\in\bar{\cN}_{v,t}\},\bar{\cW}\right), \\
  &\text{Update:}\quad h_v^{(1)} = M_{up}\left(h_v^{(0)},m_v^{(0)},\bar{\cW}\right),
\end{align}
\end{subequations}
where $M_{ag}$ and $M_{up}$ are nonlinear mappings, e.g., neural networks, and $m_v^{(0)}$ denotes the information aggregated from \CORR{$\bar{\cN}_v$, $1$-hop neighbors of node $v$.} Hence the new vector of features at node $v$, $h_v^{(1)}\in\bR^{D_1}$, is a function of the features from its 1-hop neighbors.  In other words, one MP step corresponds to the information exchange between 1-hop neighbors.  The information exchange between $k$-hop neighbors is achieved using $k$ consecutive MP steps.

\begin{example}\label{MPexample}
    One of the simplest MP step is the multiplication with the \CORR{weighted} graph Laplacian.  Denote the vectors of features at all nodes as
    \begin{equation}
        H^{(0)}=[h_1^{(0)}, h_2^{(0)}, \cdots, h_M^{(0)}]^\top \in\bR^{M\times D_0}.
    \end{equation}
 Then the aggregation and updating are, respectively,
    \begin{align*}
        m_v^{(0)} &= \sum_{j\in\bar{\cN}_{v,t}} \vL_{vj}h_j^{(0)}, \\
        h_v^{(1)} &= \vL_{vv} h_v^{(0)} + m_v^{(0)},
    \end{align*}
\CORR{which can be written in a compact form as   $H^{(1)}=\vL H^{(0)}$. Effectively, this operation is a generalization of the convolution operator on graphs, where the graph Laplacian matrix serves as the filter in the usual convolution operation. Furthermore, $\vL^k H^{(0)}$ would result in the interaction among $k$-hop neighbors and hence is equivalent to $k$ MP steps. 

A popular polynomial filter \cite{Defferrard2017}, which we will use in our numerical example, is the Chebyshev polynomial of the transformed Laplacian matrix, 
\[
T_s(\tilde{\vL}) = \sum_{i=0}^s w_i \tilde{\vL}^i,
\]
where $\{w_i\}_{i=0}^s$ are the coefficients corresponding to Chebyshev polynomials of degree-$s$, and the matrix $\tilde{\vL} = (2/\lambda_{\max})\vL-\vI$, is a transformed version of the weighted graph Laplacian, where $\lambda_{\max}$ is the largest eigenvalue of $\vL$.  The transformation ensures that the range of eigenvalues of $\tilde{\vL}$ is $[-1,1]$ and matches the domain of Chebyshev polynomials. Since the matrix $T_s(\tilde{\vL})$ contains up to $s$th power of $\tilde{\vL}$, hence it is clear that one $s$th order Chebyshev convolution, $T_s(\tilde{\vL})H^{(0)}$, accounts for the interaction of $s-$hop neighbors, which is equivalent to $s$ MP steps. 
}
\end{example}

\paragraph{Graph Convolutional Layers (GCLs).}
In this study, the MP mechanism is implemented using the GCLs with the ChebConv network \cite{Defferrard2017}, which performs the MP aggregation and updating over all nodes simultaneously.


Collecting the vectors of features at all nodes as $H^{(j)}=[h_1^{(j)}, h_2^{(j)}, \cdots, h_M^{(j)}]^\top\in\bR^{M\times D_j}$, we define the $S$th order ChebConv-based graph convolution as
\begin{equation}\label{eq:cheb1}
    H^{(j+1)} = \sigma\left( \sum_{s=0}^{S}T_s(\tilde{\vL})H^{(j)}\Theta_s^{(j)} \right) \equiv f_{S}(H^{(j)},\bar{\cG};\Theta^{(j)}),
\end{equation}
where $\sigma$ is a nonlinear element-wise activation function, and $\Theta^{(j)}=\{\Theta_s^{(j)}\in\bR^{D_j\times D_{j+1}}\}_{s=0}^S$ are learnable parameters. \CORR{In \eqref{eq:cheb1}, we have defined the map $f_{S}$ for convenience of the following discussion. Importantly, this map takes  $\bar{\mathcal{G}}$ as an input and encodes the graph information, especially the weights in \eqref{Graphweight}, in the transformed Graph Laplacian matrix, $\tilde{\vL}$, as defined in Example~\ref{MPexample}, and employs a ChebConv that performs $S$ MP steps. We should point out that $f_S$ is flexible in taking information of the coarse-grained graphs $\bar{\cG}_t$ at arbitrary time $t$ and thus allowing for modeling of time-varying topologies.

Our analysis in Section~\ref{sec3} suggests that the Graph Neural Network model should at least allow for  $2K$-hops of interactions. One way to realize it is to use the single layer model in \eqref{eq:cheb1} with $S=2K$, or to use $N_C-$layers of ChebConv network in \eqref{eq:cheb1}, 
\begin{equation}
G_{2K,t}(H^{(0)},\bar{\cG}_t) = f_S^{(N_C)}\circ \ldots \circ f_S^{(1)} (H^{(0)},\bar{\mathcal{G}}_t), \text{ where } S\times N_C=2K,\label{G2K}
\end{equation}
which is what we use in the numerical experiments in Section~\ref{sec5}. Here, we note that $G_{2K,t}$ depends on the parameters $\Theta^{(j)}$ of $f_S^{(j)}$, for $j=1,\ldots, N_C$.

For $T=1$ in \eqref{eqn_gnn}, choosing the input vectors as the feature, $H^{(0)} = \bar{x}_t$, then one can set 
\begin{eqnarray}
F_G(\bar{x}_t,\bar{\cG}_t) \equiv G_{2K,t}(\bar{x}_t,\bar{\cG}_t)\label{FG}
\end{eqnarray}
as a GNN model for $\dot{\bar{x}}$ with Markovian dynamics. Following such an approach for $T>1$ in \eqref{eqn_gnn}, unfortunately, is not computationally attractive due to the complexity induced in representing all the coarse-grained graphs $\bar{\cG}_{t-s}$, at  $s= 1, \ldots, T-1$. For example, if one considers the following GNN model,
\begin{eqnarray}
F_G(\bar{x}_t,\ldots,\bar{x}_{t-T+1},\bar{\cG}_t,\ldots,\bar{\cG}_{t-T+1}) = \sum_{i=0}^{T-1} c_i G_{2K,t-i}(\bar{x}_{t-i},\bar{\cG}_{t-i}),\label{compositeFG}
\end{eqnarray}
the total number of parameters is $T$ times more than that in \eqref{FG}. 

To avoid such a computationally expensive modeling approach, we consider employing an autoencoder that takes only the topology information of $\bar{\cG}_t$ as part of the inputs. Computationally, we assign an integer index $J$ to each distinct coarse-grained topology $\bar{\cG}$, and replace the inputs of $\{\bar{\cG}_{t},\ldots, \bar{\cG}_{t-T+1}\}$, with their corresponding indices $\{J_{t},\cdots,J_{t-T+1}\}$.
Since the inputs are of mixed type, continuous ($\bx$) and discrete ($J$), an encoding step is incorporated to map the inputs to a continuous latent space, where graph convolution steps will occur. Correspondingly, a decoding step is needed to extract the predicted $\dot{\bx}$ out of the latent space.

Certainly, replacing the topologies with indices would sacrifice connectivity information in the past $T$ steps, but the numerical results in the next section will show that such replacement achieves satisfactory accuracy in the prediction.  The direct use of all topologies, as in \eqref{compositeFG}, may become necessary when the network dynamics is more nonlinear and/or when the number of topologies grows; this is left for future study.
}


\paragraph{Detailed GNN Architecture.}
Here, we present the details of the GNN-based reduced-order model.
In total, the GNN is implemented via an Encoder-Processor-Decoder architecture \cite{Yu2024,Scarselli2009}.
The details of the three components are explained in the following,
\begin{compactenum}
    \item Encoder: First, the encoder is applied to each individual node.  It maps the sequence of partial measurements at a node and the topology indices, to a feature vector $h_i^{(0)}\in\bR^{D_0}$. For node $i$ at time step $k$, the encoder $f_E$ is
    \begin{equation}\label{eqn:encoder}
        h_i^{(0)} = f_E(\bx_{i,t},\bx_{i,t-1},\cdots,\bx_{i,t-T+1},J_{t},J_{t-1},\cdots,J_{t-T+1};\Theta^{(0)}),
    \end{equation}
    where $f_E$ is implemented as a standard fully-connected NN (FCNN) of $N_M$ hidden layers with a set of trainable parameters $\Theta^{(0)}$. After the encoding, the feature vectors of all the nodes are denoted as $H^{(0)}\in\bR^{M\times D_0}$\CORR{, where $M$ is the number of coarse-grained nodes.}

    \item Processor: Subsequently, a stack of $N_C$ graph MP layers, $\cG_{2K,p}$ in \eqref{G2K}, serves as processors that successively aggregate the features from each node and its neighbors and update the feature vectors at each node. \CORR{Denote
    \begin{equation}\label{eqn:processor}
    H^{(N_C)} = G_{2K,t}(H^{(0)},\bar{\cG}_t;\Theta^{(1)},\ldots,\Theta^{(N_C)})
    \end{equation}
    as the output feature.}
    
    
    \item Decoder: Finally, the decoder maps the feature vector of each node to the desired output at the corresponding node, i.e., the rate of change,
    \begin{equation}\label{eqn:decoder}
        \dot{\bx}_i=f_D(h_i^{(N_C)};\Theta^{(N_C+1)}),
    \end{equation}
    where $f_D$ is a FCNN of $N_M$ hidden layers with trainable parameters $\Theta^{(N+1)}$.
\end{compactenum}

\CORR{Using \eqref{eqn:encoder}-\eqref{eqn:decoder}, the specific functional form of GNN in \eqref{eqn_gnn} is,
\begin{equation}
    F_G = f_D\circ G_{2K,t}\circ f_E,
\end{equation}
where the learnable parameters are $\vtQ=[\Theta^{(0)},\Theta^{(1)},\cdots,\Theta^{(N_C+1)}]$.}

\paragraph{Implementation and Training.}
The GNN models are implemented using PyTorch Geometric (PyG) \cite{PyG2019}, an open-source machine learning framework with Graph Network architectures built upon PyTorch \cite{PyTorch2019}.

The GNN model is trained by minimizing the prediction loss at all time steps.  Given a trajectory of $N_t$ steps and a model having a time delay embedding of $T$ steps, the prediction loss is defined as
\begin{equation}
    \CORR{\cJ(\vtQ) = \sum_{t=T}^{N_t-1} \norm{ \frac{\bx_{t+1}-\bx_{t}}{\Delta t} - F_G\left(\{\bx_i\}_{i=t-T+1}^{t},\cG_t,\{J_i\}_{i=t-T+1}^{t-1};\vtQ\right) }.}
\end{equation}
The loss is minimized using the standard Adam algorithm \cite{Diederik2017}, with an exponential decay scheduling of learning rate. 

\section{Numerical Results}\label{sec5}

In this section, we demonstrate the effectiveness and versatility of the proposed model for graph dynamics with fixed and time-varying topologies.  Two examples are considered.  One is an academic example based on the Kuramoto oscillators, and the other is a more practical application based on a power grid system.

\subsection{Kuramoto Dynamics}

Based on the formulation in Example \ref{eg:kuramoto_1}, a Kuramoto system having 20 nodes is constructed.  Each node is assigned a natural frequency $\omega_i\sim U([1,15])$, where $U([a,b])$ denotes a uniform distribution over interval $[a,b]$.

The topology of the Kuramoto system is generated as follows.
The nodes are evenly divided into five groups, and within each group, the nodes are fully connected.  Each group is randomly connected to at most two other groups, and every pair of connected groups has only one edge between two randomly selected nodes from each group.  All the edge weights $\kappa_{ij}\sim U([\kappa_l,\kappa_u])$, where the values of $\kappa_l,\kappa_u$ will be specified later.  An example topology is shown in Fig. \ref{kuramoto_topo}(a), where the diagonal elements are $\omega_i$, and the off-diagonal elements are $\kappa_{ij}\sim U([4,6])$.  The choice of graph topology and edge weights encourages strong and dense intra-group interactions and weak and sparse inter-group interactions.
When time-varying topologies are considered, the intra-group connections are kept the same, but each group is randomly connected to different groups.

Next, following Example \ref{eg:kuramoto_2}, the Kuramoto topology is coarse-grained by averaging the states of each group as well as the edge weights.
\CORR{Every 4 consecutive nodes are collected as one group, so in total there are $M=5$ groups, leading to a 5-node coarse-grained graph.  An illustration of the coarse-graining is shown in Fig. \ref{kuramoto_topo}(b), where the diagonal elements represent $\alpha_i^{11}$ and the off-diagonal non-zero elements represent $\beta_{ij}^{11}$.}  Furthermore, following Example \ref{eg:kuramoto_3}, the coarsened edge weights are used in the GNN model for the prediction of the averaged states given past observations.

\insertfig{kuramoto_topo}{0.95}{An example topology before and after coarse-graining.}


The hyperparameters of the nominal GNN model are outlined in Table \ref{tbl:arch}, following the architecture specified in Sec. \ref{sec:gnn}. Every hidden layer has 128 neurons.

\begin{table}[htbp]
\centering
\caption{Detailed hyperparameters of the nominal GNN model.}
\label{tab:model_parameters}
\begin{tabular}{@{}lll@{}}
\toprule
\textbf{Component} & \textbf{Layer Type} & \textbf{Parameters} \\ \midrule
\textbf{Encoder}   & FCNN       & $N_M=1$ \\
          & Activation & PReLU     \\ \midrule
\textbf{Processor} & GCL & $N_C=2$, $S=2$  \\
          & Activation & PReLU     \\ \midrule
\textbf{Decoder}   & FCNN       & $N_M=1$ \\
          & Activation & PReLU     \\ \bottomrule
\end{tabular}\label{tbl:arch}
\end{table}

\subsubsection{Fixed Topology}\label{sec:static_topo}

\paragraph{Model Benchmark.}
First, the numerical experiments on a fixed topology are presented.
The weights shown in Fig. \ref{kuramoto_topo}(b) are used, where $\epsilon\approx 0.375$, $d\approx 2$ and hence $d\epsilon\approx 0.75$.
A dataset comprising 100 trajectories is generated using the Kuramoto dynamics \eqref{kuramoto} with random initial conditions $\theta_i(t=0)\sim U([0,2\pi])$, $i=1,2,\cdots,20$.  Each trajectory is simulated for 10 seconds with a time step size of $\Delta t=0.01s$, and then the states are coarse-grained according to \eqref{eqn:obs}.  Eighty trajectories are used as the training dataset and the rest are used as test dataset.  For prediction, the first $T$ steps of true data are used as the initial condition for the GNN model.

To benchmark the GNN model, we employ two baseline models: a Multi-Layer Perceptron (MLP) model and a Long Short-Term Memory (LSTM) model \cite{Sutskever2014}, both of which do not account for the graph topology.  The MLP model was configured by substituting the GNN layers of the processor module \eqref{eqn:processor} with \CORR{an FCNN $F_P$ having the same number of ChebConv layers as in GNN,
\begin{equation}     H^{(N_C)}=F_P(H^{(0)};\Theta^{(1)},\ldots,\Theta^{(N_C)}),
\end{equation}}
which eliminates the constraints on the topology and allows any form of interaction between any pair of nodes.

The LSTM model is a Recurrent Neural Network (RNN), with its ability to capture long-term dependencies and sequences \cite{hochreiter1997long}, making it particularly suited for time-series and sequential data analysis. In this case, a standard LSTM implementation processes the sequence of $T$ partial measurements ($\bx_t,\bx_{t-1},\dots,\bx_{t-T+1}$) on all nodes through multiple LSTM layers, followed by a linear output layer that predicts the rate of change for each node $\dot\bx$.

The hyperparameters of both baseline models are chosen to match the size of the GNN model, ensuring a fair comparison by maintaining a similar parameter count and computational complexity across all models.


\insertfig{kuramoto_fixed}{0.95}{Model prediction on test data.}

The predictive accuracy of the models is quantified with normalized root mean square error (NRMSE) defined as, \begin{equation}
    \text{NRMSE}=\frac{1}{M}\sum_{i=1}^M\frac{\sqrt{\frac{1}{N_t}\sum_{k=1}^{N_t}(\tilde{x}_{k,i}-\bx_{k,i})^2}}{\max_k(\bx_{k,i})-\min_k(\bx_{k,i})},
\end{equation} 
where $\{\tilde{x}_k \in \bR^M\}_{k=1}^{N_t}$ denotes a predicted trajectory of length $N_t$ and $\{\bx_k \in \bR^M\}_{k=1}^{N_t}$ denotes the true trajectory.


Figure \ref{kuramoto_fixed} compares the model performance of the nominal GNN model against the baseline models for a representative test case. The prediction from the GNN models closely match the ground truth in both frequencies and magnitudes over the entire prediction horizon. The model exhibits minimal error accumulation on predictions over extended time period, achieving a NRMSE of 0.027. In contrast, the two baseline models demonstrate notable deviations in their predictions. Specifically, the MLP model accurately captures the oscillation magnitudes yet exhibits phase errors in long-term forecasts, resulting in a NRMSE of 0.200. The LSTM model presents significant discrepancies in both magnitude and frequency, leading to an even higher NRMSE of 0.386. These deviations, particularly pronounced in long-term predictions, are likely exacerbated by cumulative errors. A comparison of the predictive NRMSE for all three models across the entire test dataset is presented in Table \ref{tbl:fixtopo}. Remarkably, the GNN model achieved a NRMSE that was an order of magnitude lower than the other two models, together with the lowest standard deviation. This underscores the necessity of incorporating topological information into the coarse-grained modeling of graph dynamics.

\begin{table}[ht]
\centering
\caption{Comparison of model performance on NRMSE}
\label{tbl:fixtopo}
\begin{tabular}{@{}lcc@{}}
\toprule
\textbf{NRMSE} & \textbf{Mean} & \textbf{Std.} \\ \midrule
GNN   & \textbf{0.0285} & \textbf{0.0255}              \\
MLP   & 0.2102 & 0.1053              \\
LSTM  & 0.1664 & 0.0629              \\ \bottomrule
\end{tabular}
\end{table}

\paragraph{Parametric Study.}
Next, the effect of spatiotemporal delay embedding is examined via a parametric study over the number of the MP steps and the delay embedding length $T$. The models examined in this study still follow the architecture and parameters listed in Table \ref{tbl:arch}, but with variation in $N_C$ and $S$, which results in $N_C\times S$ hops. Specifically, the 1-hop model uses $N_C=1$ and $S=1$, the 2-hop model $N_C=1$ and $S=2$, and the 4-hop model $N_C=2$ and $S=2$.  Furthermore, for each hop, seven models of different time delay lengths $T=\{10,20,30,40,50,75,100\}$ are considered.
To obtain statistically significant results, we quantify the performance of a $k$-hop model with delay length $T$ using the mean and standard deviation of the test errors from five random cases.
In each case, a topology with $d\epsilon=0.75$ is randomly generated, and the model is trained and tested according to the methodology outlined in Sec. \ref{sec:static_topo}. 

The comparison is illustrated in Fig. \ref{kuramoto_fixed_ms_de0.75}.
First, it is clear that the predictive accuracy of the 1-hop, 2-hop, and 4-hop models enhances as the delay embedding length $T$ increases and the error is approximately inversely proportional to $T$; this trend is consistent with the results of Proposition \ref{prop_tde}.
Second, the predictive accuracy is significantly improved when the MP steps increase from 1 to 2, but the 2-hop and 4-hop models have almost identical performance.  This trend is again explained by Proposition \ref{prop_tde}.  In the Kuramoto example, the full-order dynamics has $K=1$ interaction with $d\epsilon=0.75>1/2$, and thus induces non-negligible $2K=2$ hops of interaction in the coarse-grained dynamics.  Therefore, as long as a model has at least 2 MP steps, it can accurately capture the dynamics with sufficient delay embedding length.

Furthermore, the effect of $d\epsilon$ is examined using the 2-hop model and time delay lengths $T=\{10,20,30,50,75,100\}$.
The variation of $d\epsilon$ is achieved by adjusting the range of off-diagonal weights.  Choosing $\kappa_l=2$ and $\kappa_u=3$ results in $d\epsilon=0.5$, and choosing $\kappa_l=1$ and $\kappa_u=1.5$ results in $d\epsilon=0.19$.  For each combination of time delay length and $d\epsilon$, the model performance is quantified again using the mean and standard deviation of five randomly generated cases.

The comparison is shown in Fig. \ref{kuramoto_fixed_ms_de_trend}.  While in all three cases of $d\epsilon$ the model accuracy enhances as $T$ increases, the 2-hop model clearly performs better when $d\epsilon$ increases.  This trend is explained by the results of Remark \ref{rem_de}.
When $d\epsilon=0.75$, $T\sim \frac{\CORR{R^2}(1-d\epsilon)^2}{L(d\epsilon)^2}=\frac{\CORR{8}}{9L}$; when $d\epsilon=0.5$ or $0.19$, $T\sim \frac{\CORR{R^2}}{L}=\CORR{\frac{8}{L}}$.  Hence the 2-hop model with $d\epsilon=0.75$ requires almost one order of magnitude shorter length to achieve a high predictive accuracy when compared to the other two cases. 

\insertfig{kuramoto_fixed_ms_de0.75}{0.95}{Mean and standard deviation of predictive NRMSE based on 5 randomly generated topologies with $d\epsilon=0.75$.}
\insertfig{kuramoto_fixed_ms_de_trend}{0.95}{Comparison of convergence trend with increasing $d\epsilon$ with 2-hop GNN model ($N_C=1$ and $K=2$).}

\subsubsection{Time-Varying Topology}

Next, the more challenging case of time-varying topologies is presented. A set of 5 random topologies are generated, denoted as $\mathbf{G}=\left\{\cG_1,\cG_2,\cG_3,\cG_4,\cG_5\right\}$, and the model is trained with same hyperparameters as the nominal model presented in Sec. \ref{sec:static_topo}. For the dataset used in this study, each trajectory is split into multiple segments by switching to a different topology randomly selected from $\mathbf{G}$, where the system states at the end of each segment become the initial condition for the next stint to ensure a continuous trajectory. Two scenarios of different segment lengths are explored to assess the GNN model's capabilities, and a separate model is trained and tested for each scenario.

Figure \ref{kuramoto_varying} illustrates the performance of the GNN model compared to two baseline models, MLP and LSTM, in the first scenario with less frequent topology changes. This scenario is designed to evaluate the models' ability to make longer predictions after each topology change, capturing the dynamics accurately as they gradually return to synchronized oscillation. The vertical dashed lines in the figure mark the topology transitions. 
\insertfig{kuramoto_varying}{0.95}{Comparison of model performance with sparse topology changes during simulations. }

Overall, all three models show a reduction in prediction accuracy, when compared to the fixed topology cases. However, the GNN model again outperforms the baseline models by a significant margin, achieving an NRSME of 0.096, which reflects its capability to accurately reproduce both the frequency and magnitude of system dynamics over prolonged intervals post each topology alteration. Despite some discrepancies in $z_2$ and $z_3$ after the first topology change, the GNN model demonstrated a remarkable ability to mitigate these errors over time, realigning its predictions with the correct dynamics. In comparison, the MLP model's predictions suffer from a cumulative phase error over time, lacking the ability to correct itself, which results in a significantly higher NRMSE of 0.425. The LSTM model initially provides reasonable approximations of the dynamics until the first topology change, despite noticeable deviations in phase angle and magnitude. Beyond this point, its predictions significantly diverge, leading to escalating errors and a resultant NRMSE of 1.002. This comparison shows that the GNN models the dynamics in the scenario of topology changes with much greater efficacy, primarily due to its integration of topology information directly into the modeling process, enabling it to outperform the others significantly.



\insertfig{kuramoto_varying_dense}{0.95}{GNN model performance with frequent topology changes during simulations.}

In the second scenario, the topology changes occur more frequently, which is designed to test the model stability, i.e., the model's ability to remain bounded without exhibiting divergence from the true dynamics over time, even when subjected to rapid and frequent perturbations. Given that the baseline methods (MLP and LSTM) were unsuccessful in the previous case, only the GNN model is evaluated under these conditions. Remarkably, the GNN model continues to demonstrate accurate prediction capabilities and the ability to swiftly correct any deviations from true dynamics, achieving an NRMSE of 0.108. 

For each scenario, 40 trajectories are used to test the model, and the predictive performance is summarized in Table \ref{tbl:varytopo}. The GNN model performs consistently well across the two scenarios with low predictive NRMSE and standard deviation, showcasing its robustness and reliability even in dynamic and challenging scenarios.

\begin{table}[ht]
\centering
\caption{Comparison of model performance on NRMSE across different scenarios}
\label{tbl:varytopo}
\begin{tabular}{@{}lcccc@{}}
\toprule
\textbf{Scenario} & \multicolumn{2}{c}{\textbf{Few topology changes}} & \multicolumn{2}{c}{\textbf{Frequent topology changes}} \\\hline
\textbf{NRMSE}
& \textbf{Mean} & \textbf{Std.} & \textbf{Mean} & \textbf{Std.} \\
\midrule
GNN   & \textbf{0.0991} & 0.0926 & \textbf{0.1093} & 0.1428 \\
MLP   & 0.3748 & \textbf{0.0661} & - & - \\
LSTM  & 1.4855 & 0.3009 & - & - \\
\bottomrule
\end{tabular}
\end{table}

\subsection{Power System}

In this section, we consider a power system with time-varying topology to evaluate the performance of the GNN model in a real-world scenario.
The power system contains 5 generator buses and 5 load buses, and its configuration with a circuit breaker is illustrated in Fig. \ref{power_sys_topo_2}.
The dynamics is modeled using a set of DAEs that describe the dynamics of the generators, loads, and the transmission network, as discussed in \CORR{Appendix \ref{sec_pow}}.

\insertfig{power_sys_topo_2}{0.7}{The topology of the power system with 10 buses.}

The power system topology varies over time as follows. Initially, the system operates at steady state with the circuit breaker between Buses 3 and 4 closed. At $T=0.15s$, a topology change is introduced by opening the circuit breaker; this event effectively splits the power system into two isolated islands, and initiates transient dynamics within each island.  Finally, at $T=1.9s$, the circuit breaker is closed, which recovers the initial topology and initiates again some transient dynamics.
The time-varying topology is simulated by modifying the admittance matrix of the power system at the specified time instances. The opening of the circuit breaker is modeled by setting the corresponding entries in the admittance matrix to zero, while the closing is modeled by restoring the original values.
Mathematically, the transient dynamics induced by the topology changes is due to the change in the admittance matrix of the DAE for the power system.

Following Example \ref{eg:power_2}, the power system is coarse-grained using the bus voltage amplitude $|V_i|$ and phase angle $\alpha_i$ of each node.
Specifically, the observables of node $i$ are defined as $\bx_i=[|V_i|, \sin(\phi_i), \cos(\phi_i)]$.
The coarse-graining still results in a 10-node graph.  The edge weights are computed by \eqref{eqn_pow_wij} following Example \ref{eg:power_3}, where the user-specified parameter \CORR{$\eta$} is chosen so that the weights are in the range of $[0.2,1]$ \cite{Yu2024}.

Next, for the GNN modeling of the coarse-grained dynamics, we need to examine if the interaction strength between the nodes follows the power law assumption employed in Proposition \ref{prop_tde}.
To do so, we first compute the strength ratios $\gamma(k_{ij})=\beta_{ij}^{11}/\alpha_i^{11}$ between nodes $i$ and $j$, where $k_{ij}$ is the number of hops between the two nodes.  Then, we plot the ratios from all pairs of nodes against the number of hops, which are shown as a boxplot in Fig. \ref{power_sys_ratio}; note that the strength ratio at 0-hop is fixed to be unity.
It is clear that the interaction strength becomes negligible beyond 2 hops, so the power system has approximately $K=2$ hops of interactions.
To verify the power law, we fit the data with a regression model $\gamma(k)=\epsilon^k$, as shown in Fig. \ref{power_sys_ratio}.  It was found $\epsilon\approx 0.313$ and the power law assumption is confirmed with a coefficient of determination as high as $0.836$.
Lastly, the average degree of the graph is $d=1.8$, and hence the factor $d\epsilon\approx 0.5634>1/2$.  This shows that $2K=4$ hops are needed in the GNN model. 

\insertfig{power_sys_ratio}{0.8}{The ratios of interaction strengths in the fixed topology case.}

\subsubsection{Model Training and Prediction}
Based on the preceding discussion, a 4-hop GNN model is employed with the hyperparameters listed in Table \ref{tbl:arch} and the architecture detailed in Sec. \ref{sec:gnn}; 128 neurons are used in each hidden layer.

For data generation, the system is simulated for $2.7s$ with a time step of $\Delta t=0.001s$.  A total of 160 trajectories are generated for training, and 40 additional trajectories for testing.

The trained GNN model demonstrates accurate and consistent performance across all 40 test cases, with NRMSE$=0.0038\pm0.0007$, indicating the model's predictions closely match the actual measurements with minimal deviation. An example of model prediction is shown in Fig. \ref{power_system_results}, illustrating the model's exceptional accuracy in predicting both voltages and phase angles across all buses. 

The GNN model's performance is particularly impressive in capturing the initial transient dynamics triggered by topology changes at $T=0.15s$ and $T=1.9s$. During these events, the model accurately predicts the sudden changes in bus voltages and phase angles. Furthermore, the model successfully identifies the new steady-state equilibrium points reached by the system after each transient event.

\insertfig{power_system_results}{0.95}{Model prediction of bus voltage with varying topologies.}

\section{Conclusions}

In this paper, we introduced a systematic framework utilizing Graph Neural Networks (GNNs) for the non-Markovian reduced-order modeling (ROM) of coarse-grained dynamical systems represented on graphs.
These systems are characterized by heterogeneous nodal dynamics, up to $K$-hop interactions among the nodes, and potentially time-varying topologies.
For the reduced-order modeling, we consider the coarse-graining for both the nodal states and groups of nodes.
Our approach hinges on a systematic design of the GNN architecture, guided by the Mori-Zwanzig (MZ) formalism.  In particular, we inspect how the leading term of the MZ memory term depends on the coarse-grained interaction coefficients that encode the graph topology, and subsequently determine the appropriate functional form of the GNN-based ROM.

A pivotal element of our methodology is the adaptation of the GNN architecture based on the decay of interaction strength with the hop distance, governed by a power-law decay factor, $\epsilon$.
Specifically, for weaker interactions, where $d\epsilon < 1/2$ with $d$ representing the average degree, the GNN architecture requires only $K$ Message Passing (MP) steps to effectively capture $K$-hop dynamical interactions. In contrast, stronger interactions (i.e., $d\epsilon > 1/2$) necessitate at least $2K$ MP steps, underscoring the importance of tailoring the depth of the network to the dynamics of the system under study.  Moreover, we found that the memory length essential for an accurate ROM inversely correlates with the interaction strength factor $d\epsilon$, which is a significant insight for designing more compact and precise predictive models.

Guided by the theoretical finding, we constructed a GNN architecture involving an autoencoder and graph convolutions.
The convolution is based on a classical spectral-based formulation, ChebConv.  A ChebConv NN has \CORR{$N_C\times S=2K$} MP steps, where $N_C$ is the number of layers and $S$ is the Chebyshev polynomial order.
The MZ-guided GNN structure was numerically verified through two examples.
One example is a heterogeneous Kuramoto oscillator model, that has 1-hop interactions and both the nodal states and nodes are coarse-grained.  Both fixed and time-varying topologies are considered.
The GNN model shows notable improvements in accuracy over conventional MZ-based ROMs; such improvements are attributed to the coordinated incorporation of the graph topology and MP steps into the GNN architecture.
To demonstrate how the proposed methodology can be applied to real world applications, we also considered a power system model, that has multi-hop interactions and time-varying topologies.
In the ROM, only the nodal states are coarse-grained.
We started with a practical and generalizable procedure for determining the interaction strength factor $\epsilon$ from the system dynamics.
Subsequently, we chose the hyperparameters in the GNN architecture, especially the MP steps, based on the identified $\epsilon$.  Lastly, the GNN-based model constructed in this manner is demonstrated to accurately capture the intricate power system dynamics due to topology changes of the power grid.

Collectively, the findings from this study not only validate the effectiveness of our proposed GNN-based framework in modeling complex dynamical systems but also suggest its broad applicability across a variety of scientific and engineering domains. These include, but are not limited to, modeling particle systems in physics, predicting chemical reactions, and simulating fluid dynamics. The ability of GNNs to integrate and adapt to the nuances of networked systems paves the way for significant advancements in the understanding and prediction of complex systems dynamics, offering promising avenues for future research and application.

\section*{Acknowledgments}

This work was partially supported by the NSF DMS-2229435. The research of J.H. was partially supported by the NSF grants DMS-2207328 and the ONR grant N00014-22-1-2193.  

\appendix
\section{Formulation of Power System Dynamics}\label{sec_pow}

\subsection{System-level dynamics}

The transient dynamics of a power system is governed mainly by DERs and power loads. The modeling of the system of $N$ nodes is typically represented by a set of nonlinear differential-algebraic equations (DAEs) when the power-electronic interfaces of DERs are modeled by an averaged dynamics equation,
\begin{subequations}
\begin{align}  
\dot{\mathbf{x}}&=\boldsymbol{f}\big(\mathbf{x},\mathbf{y}\big), \label{eq_DAE_d}\\
\mathbf{0}&=\boldsymbol{g}\big(\mathbf{x},\mathbf{y}\big) \label{eq_DAE_a}, 
\end{align} 
\end{subequations}
where the differential equations $\boldsymbol{f}$ summarize the model of transient dynamics from DERs; 
the algebraic equations $\boldsymbol{g}$ represent the power flow equations essential for compliance from the viewpoint of the interconnected system;
$\mathbf{x} = [\mathbf{x}_{1},\mathbf{x}_{2},\cdots,\mathbf{x}_{N}]$ is the state variables of DER units, e.g., $\mathbf{x}_{i}$ is the state variables of the $i^{\text{th}}$ DER unit;  
$\mathbf{y} = [\mathbf{y}_1,\mathbf{y}_2,\cdots,\mathbf{y}_N]$ is the algebraic variables, e.g.,  bus voltage amplitude and angle of the interconnected system. 
In the following, \eqref{eqn_node_x_0} gives an example of a part of \eqref{eq_DAE_d} that is associated with the $i^{\text{th}}$ DER unit,
\begin{equation}\label{eqn_node_x_0}
    \dot{\vx}_i = \vA_i\vx_i + \vf_i(\vx_i,\vy_{i},\vI_{gi}),
\end{equation}
where $\vy_{i}$ is the voltage of the node connecting to the $i^{\text{th}}$ DER unit, and $\vI_{gi}$ is its corresponding current output, as explained in \eqref{eq_current}.  An instance of \eqref{eqn_node_x_0} is provided in \ref{sec_app_node}.

Since \eqref{eq_DAE_a} represents the power flow equations, from the viewpoint of current,
\begin{equation} \label{eq_current}
    \vY \vy = \vI_{g} - \vI_{l},
\end{equation}
where $\vY$ is the admittance matrix of the system, $\vI_{g}=[\cdots, \vI_{gi}, \cdots]$, and $\vI_{l}=[\cdots, \vI_{li}, \cdots]$. Here, $\vI_{gi}$ is non-zero if node $i$ has a DER unit. Since $\vI_{gi}$ is the current output of the $i^{\text{th}}$ DER unit, it is a function of $\vx_{i}$ and $\vy_{i}$,
defined as, $\vI_{gi}=\vh_i(\vx_i)- Y_i\vy_i$, where $\vh_i(\vx_i)$ represents the current related to the output voltage of the $i^{\text{th}}$ DER unit and
$Y_i $ represents the admittance of the connection circuit of the $i^{\text{th}}$ DER unit and is a constant value. In \eqref{eq_current}, $\vI_{li}$ is non-zero if node $i$ has a load, which could be a constant, a linear function of $\vy_i$, or inversely proportional to $|\vy_i|^2$.  In the current study, we only consider the first two cases of $\vI_{li}$.

Based on above discussion, $\vI_{gi}$ is a function of $\vx_i$ and $\vy_i$, and hence \eqref{eqn_node_x_0} can be written more compactly as,
\begin{equation}
    \dot{\vx}_i = \vA_i\vx_i + \hat{\vf}_i(\vx_i,\vy_i) \equiv \tilde{\vF}_i(\vx_i,\vy_i).
\end{equation}

Collectively, denote the nodes that have DERs as $\cR$ and the nodes that have loads as $\cS$.  Note that it is possible that $\cR\cap\cS\neq\varnothing$.  Then, we can write the currents of all nodes as
\begin{equation} \label{eq_current2}
    \vI = \vI_g - \vI_l \equiv \vh(\vx) - \hat{\vY} \vy- (\vM\vy+\vm),
\end{equation}
where
$$
\vh(\vx) = [\vh_1(\vx_1)^{\mathsf{T}},\ \vh_2(\vx_2)^{\mathsf{T}},\ \cdots, \vh_N(\vx_N)^{\mathsf{T}}]^{\mathsf{T}},
$$
and $\vh_j=0$ if $j\notin\cR$; 
$\hat{\vY} $ is a diagonal matrix and  the $j^{\text{th}}$ diagonal block is 0 if $j\notin\cR$;
$\vM$ is a diagonal matrix and the $i^{\text{th}}$ diagonal block of $\vM$ and the $i^{\text{th}}$ block of $\vm$ are 0 if $i\notin\cS$.
When we only consider the first two cases of $\vI_{li}$ in the current study, based on \eqref{eq_current} and \eqref{eq_current2},
\begin{equation} \label{eq_current3}
    \tilde{\vY} \vy = \vh(\vx) - \vm,
\end{equation}
where $\tilde{\vY}=\vY+\hat{\vY}+\vM$ is the extended admittance matrix.

From \eqref{eq_current3}, take time derivative on both sides,
\begin{equation}
    \tilde{\vY}  \dot{\vy} = \nabla_\vx\vh(\vx)\dot{\vx}, 
\end{equation}
or element-wise
\begin{equation}\label{eqn_node_v}
    \dot{\vy}_i = \sum_{j=1}^N \tilde{\vY}^{-1}_{ij} \dot{\vI}_j = \sum_{j=1}^N \tilde{\vY}^{-1}_{ij} \nabla_\vx \vh_j \dot{\vx}_j,
\end{equation}
where $\tilde{\vY}^{-1}_{ij}$ is the $(i,j)$th block of $\tilde{\vY}^{-1}$.

Then for the nodes that connect to DERs, the states are $[\vx,\vy]$, and the dynamics is,
\begin{equation}\label{eqn_extended_states}
    \begin{bmatrix}
        \dot{\vx}_i \\ \dot{\vy}_i
    \end{bmatrix} =
    \begin{bmatrix}
        \vI & O \\ O & \tilde{\vY}^{-1}_{ii}
    \end{bmatrix}
    \begin{bmatrix}
        \tilde{\vF}_i(\vx_i,\vy_i) \\ \nabla_\vx \vh_i \tilde{\vF}_i(\vx_i,\vy_i)
    \end{bmatrix}
    +
    \sum_{i\neq j}
    \begin{bmatrix}
        O & O \\ O & \tilde{\vY}^{-1}_{ij}
    \end{bmatrix}
    \begin{bmatrix}
        O \\ \nabla_\vx \vh_j \tilde{\vF}_j(\vx_j,\vy_j)
    \end{bmatrix},
\end{equation}
where $\vI$ is identity matrix.

For a non-DER node, the dynamics is,
\begin{equation}
    \dot{\vy}_i = \sum_{j\in\cR} \tilde{\vY}^{-1}_{ij} \nabla_\vx \vh_j \tilde{\vF}_j(\vx_j,\vy_j).
\end{equation}
\CORR{Note that the node states do not directly interact with each other.  Instead the interaction is purely from the inverse of $\tilde{\vY}$.

Finally, if we normalize $\tilde{\vF}_i$ with a factor $K_{f,i}$ and $\nabla_\vx\vh_i$ with a factor $K_{h,i}$, then the terms in \eqref{eqn_pow_der} and \eqref{eqn_pow_nde} are obtained,
$$
\vK = K_{f,i}\vI,\quad \bar{\vY}^{-1}_{ii}= K_{f,i}K_{h,i}\tilde{\vY}^{-1}_{ii},\quad \vF_i=\tilde{\vF}_i/K_{f,i},\quad \vG_i=\nabla_\vx \vh_i \tilde{\vF}_i/(K_{f,j}K_{h,j}).
$$}

\subsection{Node-level dynamics}\label{sec_app_node}

The DER connection circuit is shown in Fig.~\ref{fig_connection_diagram}, where the controllers for grid-forming (Battery and Micro-Turbine in Fig.~\ref{power_sys_topo_2}) and grid-following (PV and Fuel Cell in Fig.~\ref{power_sys_topo_2}) units are shown in Fig.~\ref{fig_Vf_diagram} and Fig.~\ref{fig_PQ_diagram}, respectively.
\begin{figure}
    \centering
    \includegraphics[width=0.8\textwidth]{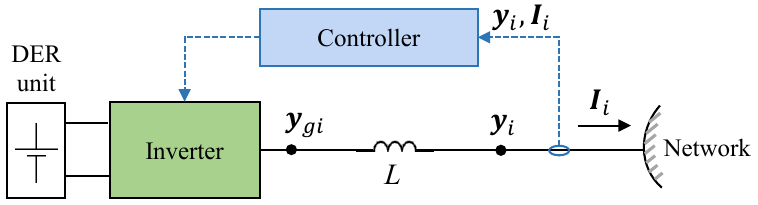}
    \caption{The connection circuit of DER}
    \label{fig_connection_diagram}
\end{figure}
\begin{figure}
    \centering
    \includegraphics[width=0.9\textwidth]{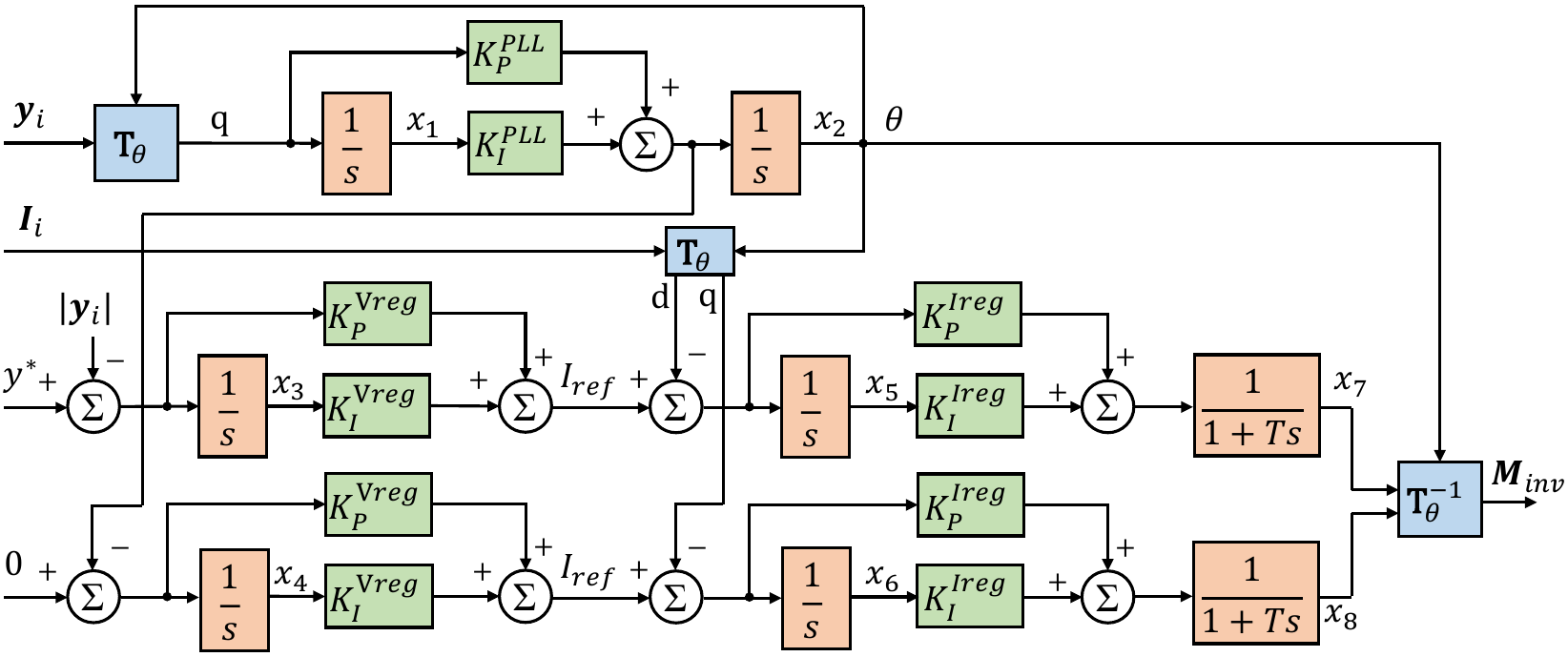}
    \caption{Grid-forming controller}
    \label{fig_Vf_diagram}
\end{figure}

For instance, the dynamics of the grid-forming controllers in Fig.~\ref{fig_Vf_diagram} can be described in the following equations, where the bus voltage is defined as $\mathbf{y}_i=y_i\angle \alpha_i=y_ie^{j\alpha_i}$, and the current is $\mathbf{I}_i=I_i\angle \beta_i=I_ie^{j\beta_i}$. Correspondingly, the modeling of grid-following controllers can also be developed,
\begin{subequations}
\begin{align}
    \dot{x}_1 &= -\abs{y_i}\cos\angle \alpha_i \sin{x_2} + \abs{y_i}\sin\angle \alpha_i \cos{x_2} \label{E:Vf 1}\\
    \dot{x}_2 &= K_P^{PLL} (-\abs{y_i}\cos\angle \alpha_i \sin{x_2} + \abs{y_i}\sin\angle \alpha_i \cos{x_2}) + K_I^{PLL} x_1 \\
    \dot{x}_3 &= y^* - \abs{y_i} \\
    \dot{x}_4 &= -K_P^{PLL} (-\abs{y_i}\cos\angle \alpha_i \sin{x_2} +\abs{y_i}\sin\angle \alpha_i \cos{x_2}) - K_I^{PLL} x_1 \\
    \dot{x}_5 &= K_P^{Vreg} (y^* -  \abs{y_i}) + K_I^{Vreg} x_3 - (\abs{I_i} \cos \beta_i \cos{x_2} + \abs{I_i} \sin \beta_i \sin{x_2}) \\
    \dot{x}_6 &= K_P^{Vreg} \big[-K_P^{PLL} (-\abs{y_i}\cos\angle \alpha_i \sin{x_2} + \abs{y_i}\sin\angle \alpha_i  \cos{x_2}) \notag \\
    & \;\, - K_I^{PLL}x_1\big] + K_I^{Vreg} x_4 - (-\abs{I_i} \cos \beta_i \sin{x_2} + \abs{I_i} \sin \beta_i \cos{x_2}) \\
    \dot{x}_7 &= \frac{1}{T}\big[K_I^{Ireg}x_5+K_P^{Ireg}(K_P^{Vreg} (y^* -  \abs{y_i}) + K_I^{Vreg} x_3 \notag \\
    & \;\, - (\abs{I_i} \cos \beta_i \cos{x_2} + \abs{I_i} \sin \beta_i \sin{x_2})) -x_7\big]\\
    \dot{x}_8 &= \frac{1}{T}\big[K_I^{Ireg}x_6+K_P^{Ireg}(K_P^{Vreg} \big[-K_P^{PLL} (-\abs{y_i}\cos\angle \alpha_i \sin{x_2} + \abs{y_i}\sin\angle \alpha_i  \cos{x_2}) \notag \\
    & \;\, - K_I^{PLL}x_1\big] + K_I^{Vreg} x_4 - (-\abs{I_i} \cos \beta_i \sin{x_2} + \abs{I_i} \sin \beta_i \cos{x_2}) \big].
\end{align}    
\end{subequations}

Then, we can rewrite the above equations in the following matrix format,
\begin{equation}
\begin{aligned}
    &\begin{bmatrix}
        1 & 0 & 0 & 0 & 0 & 0 & 0 & 0\\
        -K_P^{PLL} & 1 & 0 & 0 & 0 & 0& 0& 0 \\
        0 & 0 & 1 & 0 & 0 & 0 & 0& 0\\
        K_P^{PLL} & 0 & 0 & 1 & 0 & 0& 0& 0 \\
        0 & 0 & -K_P^{Vreg} & 0 & 1 & 0 & 0& 0\\
        0 & 0 & 0 & -K_P^{Vreg} & 0 & 1& 0& 0\\
        0 & 0 & 0 & 0 & -K_P^{Ireg} & 0& 1& 0\\
        0 & 0 & 0 & 0 & 0 & -K_P^{Ireg}& 0& 1
    \end{bmatrix}
    \dot{\mathbf{x}}_i \\
    &=
    \begin{bmatrix}
        0 & 0 & 0 & 0 & 0 & 0 & 0& 0\\
        K_I^{PLL} & 0 & 0 & 0 & 0 & 0 & 0& 0\\
        0 & 0 & 0 & 0 & 0 & 0 & 0& 0\\
        - K_I^{PLL} & 0 & 0 & 0 & 0 & 0 & 0& 0\\
        0 & 0 & K_I^{Vreg} & 0 & 0 & 0 & 0& 0\\
        0 & 0 & 0 & K_I^{Vreg} & 0 & 0 & 0& 0\\
        0 & 0 & 0 & 0 & \frac{K_I^{Ireg}}{T} & 0 & -\frac{1}{T}& 0\\
        0 & 0 & 0 & 0 & 0 & \frac{K_I^{Ireg}}{T} & 0& -\frac{1}{T}
    \end{bmatrix}\mathbf{x}_i\\
    &+
    \begin{bmatrix}
        - \abs{y_i}\sin (\alpha_i - x_2) \\
        0 \\
        y^* - \abs{y_i} \\
        0 \\
        - \abs{I_i} \cos (\beta_i-x_2) \\
        - \abs{I_i} \cos (\beta_i-x_2)\\
        0 \\
        0 
    \end{bmatrix}
    \end{aligned},
\end{equation}
where $\mathbf{x}_i=[x_1,x_2,\cdots,x_8]^{\mathsf T}$.

\begin{figure}
    \centering
    \includegraphics[width=0.9\textwidth]{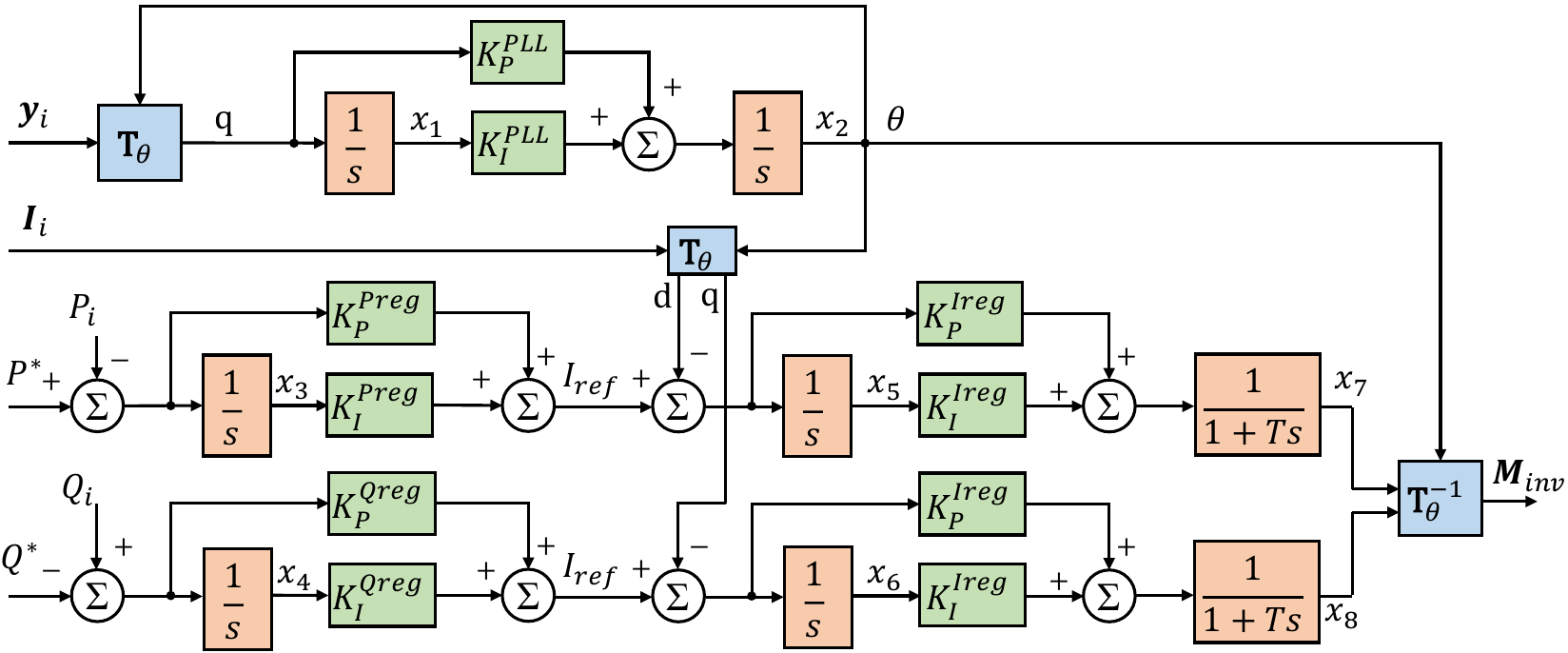}
    \caption{Grid-following controller}
    \label{fig_PQ_diagram}
\end{figure}
\section{Technical details in coarse-grained projection}

\CORR{
\subsection{Construction of left inverses}\label{sec_lin_inv}

In Sec. \ref{sec_cg}, we have defined $\Phi\in\bR^{N_x\times n}$ as a $N\times M$ block matrix and $[\Phi]_{ij}=\Phi_{ij}$ if $i\in\cV_j$ otherwise $[\Phi]_{ij}=0$; similarly for $\Psi\in\bR^{N_x\times (N_x-n)}$.  Denote $|\Phi|_2=R$.  If for a particular choice, $|\tilde{\Psi}|_2=R'$, one can choose $\Psi = \frac{R}{R'}\tilde{\Psi}$ to ensure $|\Phi|_2=|\Psi|_2=R$, as required by the main text.

Next we first construct the left inverse $\Phi^+$ for $\Phi$, such that $\Phi^+\Phi=I$ and $\Phi^+\Psi=O$.

Suppose $\cV_i=\{j_1,j_2,\cdots,j_k\}$, $k=|\cV_i|$, and define
$$
\Phi_i = \begin{bmatrix}
    \Phi_{j_1i} \\ \Phi_{j_2i} \\ \vdots \\ \Phi_{j_ki}
\end{bmatrix}.
$$
Since we require that $\Phi_i$ is a full column rank matrix, one can compute
\begin{align*}
    \Phi_i^+ &= (\Phi_i^\top\Phi_i)^{-1}\Phi_i^\top \\
    &= (\Phi_i^\top\Phi_i)^{-1} [\Phi_{j_1i}^\top, \Phi_{j_2i}^\top, \cdots, \Phi_{j_ki}^\top] \\
    &\equiv [\Phi_{ij_1}^+,\Phi_{ij_2}^+,\cdots,\Phi_{ij_k}^+],
\end{align*}
and it is clear that $\Phi_i^+\Phi_i=I$.

Next, define the left inverse $\Phi^+$ as a $M\times N$ block matrix, such that
$$
[\Phi^+]_{ij} = \left\{\begin{array}{ll}
    \Phi_{ij}^+,&\quad j\in\cV_i \\
    0 &\quad \text{otherwise.}
\end{array}\right.
$$
One can verify that $\Phi^+\Phi=I$, since its $(i,j)$th block is
$$
[\Phi^+\Phi]_{ij} = \sum_{k=1}^N \Phi_{ik}^+ \Phi_{kj} = \sum_{k\in\cV_i} \Phi_{ik}^+ \Phi_{kj} = \left\{\begin{array}{ll}
    I, & i=j \\
    O, & i\neq j
\end{array}\right.
$$
and that $\Phi^+\Psi = O$, since its $(i,j)$th block is
$$
[\Phi^+\Psi]_{ij} = \sum_{k=1}^N \Phi_{ik}^+ \Psi_{kj} = \sum_{k\in\cV_i} \Phi_{ik}^+ \Psi_{kj} = (\Phi_i^\top\Phi_i)^{-1}\sum_{k\in\cV_i} \Phi_{ki}^\top \Psi_{kj} = O.
$$
Hence $\Phi^+$ is indeed the desired left inverse for $\Phi$.

Via a similar procedure, one can construct the left inverse $\Psi^+$ that shares the same sparsity pattern as $\Phi^+$.
}

\CORR{
\subsection{Projection of $A$ matrices}\label{sec_lin_dia}

During the coarse-graining, e.g., in \eqref{proj_g2}, we encounter terms such as $A^{11}=\Phi^+ A\Phi$, where $A$ is a block-diagonal matrix.  We show that $A^{11}$ is also block-diagonal given the choice of $\Phi^+$ above.

Consider the $(i,j)$th block of $A^{11}$,
$$
[A^{11}]_{ij} = \sum_{k,l=1}^N\Phi_{ik}^+ A_{kl} \Phi_{lj} = \sum_{k=1}^N\Phi_{ik}^+ \alpha_{k} \Phi_{kj} = \sum_{k\in\cV_i}\Phi_{ik}^+ \alpha_{k} \Phi_{kj},
$$
where $A_{kl}$ denotes the $(k,l)$th block of $A$, which is zero if $k\neq l$.
When $k\in\cV_i$ and since $\cV_i \cap \cV_j = \emptyset$, for $i\neq j$, $\Phi_{kj}=0$ and $[A^{11}]_{ij}=0$, hence $A^{11}$ is block-diagonal.  When $i=j$, we have denoted,
$$
\alpha_i^{11} = \sum_{k\in\cV_i}\Phi_{ik}^+ \alpha_{k} \Phi_{ki}.
$$

By a similar procedure, in \eqref{proj_g3} we can show that $A^{12}$ is block-diagonal and have defined
$$
\alpha_i^{12} = \sum_{k\in\cV_i}\Phi_{ik}^+ \alpha_{k} \Psi_{ki}.
$$
}

\CORR{
\subsection{Projection of interaction terms}\label{sec_lin_oti}

First, we elaborate how \eqref{eqn_ph1} is obtained using \eqref{eqn_prj_p}.  To simplify the notation, we drop $t$ in $B(t)$.

Note that $H_1(x,x)=\Phi^+(B\otimes F(x,x))$ is a $M\times 1$ block vector, and the $i$th block is
\begin{equation}\label{eqn_bif}
[H_1(x,x)]_i = \sum_{k\in \cV_i}\Phi_{ik}^+\sum_{j\in \cN_{k,t}^{(K)}} \beta_{kj} f_{kj}(x_k,x_j) \equiv \sum_{k\in\cV_i}\underbrace{\Phi_{ik}^+}_{m_i \times n_k} \underbrace{B_k}_{n_k \times N_x} \underbrace{F_k(x,x)}_{N_x},
\end{equation}
where $B_k=[\beta_{k1},\beta_{k2},\cdots,\beta_{kN}]\in\bR^{n_k\times N_x}$ and $F_k^\top=[f_{k1}^\top,f_{k2}^\top,\cdots,f_{kN}^\top]\in\bR^{N_x}$, and note that by definition $\beta_{kj}\neq 0$ only when $j \in \cN_{k,t}^{(K)}$.

Each term in \eqref{eqn_bif} satisfies the form assumed in \eqref{eqn_prj_p}.  After projection the $i$th term is,
\begin{align*}
[\cP H_1(x,x)]_i &= \sum_{k\in \cV_i}\Phi_{ik}^+ B_k \Phi \Phi^+ \bar{F}_k(\bar{x}_i,\bar{x}) \\
&= \sum_{k\in \cV_i}\sum_{j\in \bar{\cN}_{i,t}^{(K)} }  [\Phi_{ik}^+ B_k \Phi]_j [\Phi^+ \bar{F}_k(\bar{x}_i,\bar{x})]_j 
\\
&\equiv \sum_{k\in \cV_i}\sum_{j\in \bar{\cN}_{i,t}^{(K)} }\beta_{ijk}^{11}  \overline{f^\Phi_{kj}}(\bar{x}_i,\bar{x}_j),
\end{align*}
where note that the summation is over $\bar{\cN}_{i,t}^{(K)}$ due to coarse-graining, and
\begin{eqnarray}
\beta_{ijk}^{11} &\equiv& [\Phi_{ik}^+ B_k \Phi]_j  = \sum_{\ell\in \cV_j}\Phi_{ik}^+ \beta_{k\ell} \Phi_{\ell j},\notag \\
\overline{f^\Phi_{kj}}(\bar{x}_i,\bar{x}_j) &\equiv & [\Phi^+ \bar{F}_{k}(\bar{x}_i,\bar{x})]_j =  \sum_{\ell\in \cV_j} \Phi_{j\ell}^+ \bar{f}_{kj}(\bar{x}_i,\bar{x}_j),\notag
\end{eqnarray}
for any $k\in \cV_i$.

Subsequently, we elaborate how \eqref{proj_g34} is obtained using \eqref{eqn_prj_q}.  We have
\begin{align}
[\cQ H_1(x,x)]_i &= \sum_{k\in \cV_i}\Phi_{ik}^+ B_k \Phi \Phi^+ (F_k(x_k,x) - \bar{F}_k(\bar{x}_i,\bar{x})) + \sum_{k\in \cV_i}\Phi_{ik}^+ B_k \Psi \Psi^+ F_k(x_k,x)\notag \\
&= \sum_{k\in \cV_i}\sum_{j\in \bar{\cN}_{i,t}^{(K)} } \Big( [\Phi_{ik}^+ B_k \Phi]_{j} [\Phi^+ (F_k(x_k,x) - \bar{F}_k(\bar{x}_i,\bar{x}))]_{j}
 \notag\\  & \hspace{1in} + [\Phi_{ik}^+ B_k \Psi]_j [\Psi^+ (F_k(x_k,x)]_j \Big) \label{mixdecomp} \\
&\equiv \sum_{k\in \cV_i}\sum_{j\in \bar{\cN}_{i,t}^{(K)} } \left(\beta_{ijk}^{11}  \widetilde{f^\Phi_{kj}}(x_k,\bar{x}_i,x_{\cV_j}) + \beta_{ijk}^{12} f^\Psi_{kj}(x_k,x_{\cV_j}) \right),\notag
\end{align}
where we have defined,
\begin{eqnarray}
\beta^{12}_{ijk} &=& [\Phi_{ik}^+ B_k \Psi]_{j}  = \sum_{\ell\in \cV_j}\Phi_{ik}^+ \beta_{k\ell} \Psi_{\ell j},\notag \\  
\widetilde{f_{kj}^\Phi} 
(x_k,\bar{x}_i,x_{\cV_j}) &=&[\Phi^+ (F_k(x_k,x_\ell)-\bar{F}_{k}(\bar{x}_i,\bar{x}))]_{j} =  \sum_{\ell\in \cV_j}\Phi_{j\ell}^+ (f_{k\ell}(x_k,x_\ell)) - \bar{f}_{k\ell}(\bar{x}_i,\bar{x}_j)),\notag  \\
f^\Psi_{kj}(x_k,x_{\cV_j}) &=& [\Psi^+ (F_k(x_k,x)]_j = \sum_{\ell\in \cV_j} \Psi^+_{j\ell}f_{k\ell}(x_k,x_\ell),\notag
\end{eqnarray}
for any $k\in \cV_i, \ell \in \cV_j$. In the above, we have defined $x_{\cV_j} = \{x_\ell: \ell \in \cV_j\}$. We should point out that the summation over the set of coarse-grained $K-$hop neighbors of node-$i$, $\mathcal{N}_{i,t}^{(K)}$, is valid even for those terms involving $\Psi$ in \eqref{mixdecomp} since $\Psi$ has the same block structure as $\Phi$, as defined in Section~\ref{sec:CGG}. 

}

\bibliographystyle{plain}
\bibliography{ref}

\begin{thebibliography}{10}

\bibitem{battaglia2016interaction}
Peter Battaglia, Razvan Pascanu, Matthew Lai, Danilo Jimenez~Rezende, et~al.
\newblock Interaction networks for learning about objects, relations and
  physics.
\newblock {\em Advances in neural information processing systems}, 29, 2016.

\bibitem{chk:00}
A.J. Chorin, O.H. Hald, and R.~Kupferman.
\newblock Optimal prediction and the {Mori--Zwanzig} representation of
  irreversible processes.
\newblock {\em Proceedings of the National Academy of Sciences},
  97(7):2968--2973, 2000.

\bibitem{belbute2020}
Filipe De~Avila Belbute-Peres, Thomas Economon, and Zico Kolter.
\newblock Combining differentiable {PDE} solvers and graph neural networks for
  fluid flow prediction.
\newblock In Hal~Daumé III and Aarti Singh, editors, {\em Proceedings of the
  37th International Conference on Machine Learning}, volume 119 of {\em
  Proceedings of Machine Learning Research}, pages 2402--2411. PMLR, 13--18 Jul
  2020.

\bibitem{Defferrard2017}
Michaël Defferrard, Xavier Bresson, and Pierre Vandergheynst.
\newblock Convolutional neural networks on graphs with fast localized spectral
  filtering.
\newblock {\em arXiv:1606.09375}, 2017.

\bibitem{do2019graph}
Kien Do, Truyen Tran, and Svetha Venkatesh.
\newblock Graph transformation policy network for chemical reaction prediction.
\newblock In {\em Proceedings of the 25th ACM SIGKDD international conference
  on knowledge discovery \& data mining}, pages 750--760, 2019.

\bibitem{PyG2019}
Matthias Fey and Jan~E. Lenssen.
\newblock Fast graph representation learning with {PyTorch Geometric}.
\newblock In {\em ICLR Workshop on Representation Learning on Graphs and
  Manifolds}, 2019.

\bibitem{Filatrella2008}
Giovanni Filatrella, Arne~Hejde Nielsen, and Niels~Falsig Pedersen.
\newblock Analysis of a power grid using a kuramoto-like model.
\newblock {\em The European Physical Journal B}, 61:485--491, 2008.

\bibitem{Hamilton2020}
William~L Hamilton.
\newblock {\em Graph representation learning}.
\newblock Morgan \& Claypool Publishers, 2020.

\bibitem{harlim2015parametric}
John Harlim and Xiantao Li.
\newblock Parametric reduced models for the nonlinear schr{\"o}dinger equation.
\newblock {\em Physical Review E}, 91(5):053306, 2015.

\bibitem{hochreiter1997long}
Sepp Hochreiter and J{\"u}rgen Schmidhuber.
\newblock Long short-term memory.
\newblock {\em Neural computation}, 9(8):1735--1780, 1997.

\bibitem{Diederik2017}
Diederik~P. Kingma and Jimmy Ba.
\newblock Adam: A method for stochastic optimization, 2017.

\bibitem{kuramoto1975international}
Yoshiki Kuramoto.
\newblock Self-entrainment of a population of coupled non-linear oscillators.
\newblock In {\em International Symposium on Mathematical Problems in
  Theoretical Physics}, volume~39, page 420. Springer-Verlag, New York, 1975.

\bibitem{lasota2013chaos}
Andrzej Lasota and Michael~C Mackey.
\newblock {\em Chaos, fractals, and noise: stochastic aspects of dynamics},
  volume~97.
\newblock Springer Science \& Business Media, 2013.

\bibitem{lei2016data}
Huan Lei, Nathan~A Baker, and Xiantao Li.
\newblock Data-driven parameterization of the generalized langevin equation.
\newblock {\em Proceedings of the National Academy of Sciences},
  113(50):14183--14188, 2016.

\bibitem{mauroy2016global}
Alexandre Mauroy and Igor Mezi{\'c}.
\newblock Global stability analysis using the eigenfunctions of the koopman
  operator.
\newblock {\em IEEE Transactions on Automatic Control}, 61(11):3356--3369,
  2016.

\bibitem{mauroy2020introduction}
Alexandre Mauroy, Yoshihiko Susuki, and Igor Mezi{\'c}.
\newblock Introduction to the koopman operator in dynamical systems and control
  theory.
\newblock {\em The koopman operator in systems and control: concepts,
  methodologies, and applications}, pages 3--33, 2020.

\bibitem{Menara2019}
Tommaso Menara, Giacomo Baggio, Danielle~S Bassett, and Fabio Pasqualetti.
\newblock Stability conditions for cluster synchronization in networks of
  heterogeneous kuramoto oscillators.
\newblock {\em IEEE Transactions on Control of Network Systems}, 7(1):302--314,
  2019.

\bibitem{mori:65}
H.~Mori.
\newblock Transport, collective motion, and {Brownian} motion.
\newblock {\em Prog. Theor. Phys.}, 33:423 -- 450, 1965.

\bibitem{PyTorch2019}
Adam Paszke, Sam Gross, Francisco Massa, Adam Lerer, James Bradbury, Gregory
  Chanan, Trevor Killeen, Zeming Lin, Natalia Gimelshein, Luca Antiga, Alban
  Desmaison, Andreas Kopf, Edward Yang, Zachary DeVito, Martin Raison, Alykhan
  Tejani, Sasank Chilamkurthy, Benoit Steiner, Lu~Fang, Junjie Bai, and Soumith
  Chintala.
\newblock Pytorch: An imperative style, high-performance deep learning library.
\newblock In H.~Wallach, H.~Larochelle, A.~Beygelzimer, F.~d\textquotesingle
  Alch\'{e}-Buc, E.~Fox, and R.~Garnett, editors, {\em Advances in Neural
  Information Processing Systems 32}, pages 8024--8035. Curran Associates,
  Inc., 2019.

\bibitem{Scarselli2009}
Franco Scarselli, Marco Gori, Ah~Chung Tsoi, Markus Hagenbuchner, and Gabriele
  Monfardini.
\newblock The graph neural network model.
\newblock {\em IEEE Transactions on Neural Networks}, 20(1):61--80, 2009.

\bibitem{stinis2007higher}
Panagiotis Stinis.
\newblock Higher order mori--zwanzig models for the euler equations.
\newblock {\em Multiscale Modeling \& Simulation}, 6(3):741--760, 2007.

\bibitem{Sutskever2014}
Ilya Sutskever, Oriol Vinyals, and Quoc~V. Le.
\newblock Sequence to sequence learning with neural networks.
\newblock In {\em Proceedings of the 27th International Conference on Neural
  Information Processing Systems - Volume 2}, NIPS'14, page 3104–3112,
  Cambridge, MA, USA, 2014. MIT Press.

\bibitem{wu2020comprehensive}
Zonghan Wu, Shirui Pan, Fengwen Chen, Guodong Long, Chengqi Zhang, and S~Yu
  Philip.
\newblock A comprehensive survey on graph neural networks.
\newblock {\em IEEE transactions on neural networks and learning systems},
  32(1):4--24, 2020.

\bibitem{Yu2024}
Yin Yu, Xinyuan Jiang, Daning Huang, Yan Li, Meng Yue, and Tianqiao Zhao.
\newblock Pidgeun: Graph neural network-enabled transient dynamics prediction
  of networked microgrids through full-field measurement.
\newblock {\em IEEE Access}, 12:49464--49475, 2024.

\bibitem{zhou2020graph}
Jie Zhou, Ganqu Cui, Shengding Hu, Zhengyan Zhang, Cheng Yang, Zhiyuan Liu,
  Lifeng Wang, Changcheng Li, and Maosong Sun.
\newblock Graph neural networks: A review of methods and applications.
\newblock {\em AI open}, 1:57--81, 2020.

\bibitem{zwanzig:73}
R.~Zwanzig.
\newblock Nonlinear generalized {Langevin} equations.
\newblock {\em J. Stat. Phys.}, 9:215 -- 220, 1973.

\end{thebibliography}

\end{document}